\documentclass[12pt, reqno]{article}
\usepackage{amsmath}
\usepackage{amssymb}
\usepackage{latexsym,bm}
\usepackage{amsthm}
\usepackage{graphicx}

\usepackage[top=2.4cm,bottom=2cm,left=2.4cm,right=2.4cm]{geometry}
\usepackage{algorithm}
\usepackage{algorithmic}
\usepackage{cite}
\usepackage[colorlinks,urlcolor=blue]{hyperref}
\allowdisplaybreaks
\usepackage{dsfont}
\usepackage{multirow}
\usepackage{threeparttable}
\usepackage{slashbox}

\newtheorem{thm}{Theorem}[section]

\newtheorem{lem}{Lemma}[section]

\theoremstyle{definition}
\newtheorem{definition}{Definition}[section]
\theoremstyle{remark}

\numberwithin{equation}{section}
\theoremstyle{example}

\numberwithin{equation}{section} \numberwithin{algorithm}{section}

\begin{document}

\begin{center}

\textbf{\large A general framework for solving convex optimization problems involving the sum of three convex functions}

\end{center}

\begin{center}
Yu-Chao Tang$^{1}$, Guo-Rong Wu$^{2}$, Chuan-Xi Zhu$^{1}$
\end{center}

\begin{center}
1. Department of Mathematics, Nanchang University, Nanchang 330031,
P.R. China\\
2. Department of Radiology and BRIC, University of North Carolina at Chapel Hill, NC 27599, USA

\end{center}

\bigskip

\begin{abstract}
In this paper, we consider solving a class of convex optimization problem which minimizes the sum of three convex functions $f(x)+g(x)+h(Bx)$, where $f(x)$ is differentiable with a Lipschitz continuous gradient, $g(x)$ and $h(x)$ have a closed-form expression of their proximity operators and $B$ is a bounded linear operator. This type of optimization problem has wide application in signal recovery and image processing. To make full use of the differentiability function in the optimization problem, we take advantage of two operator splitting methods: the forward-backward splitting method and the three operator splitting method. In the iteration scheme derived from the two operator splitting methods, we need to compute the proximity operator of $g+h \circ B$ and $h \circ B$, respectively. Although these proximity operators do not have a closed-form solution in general, they can be solved very efficiently. We mainly employ two different approaches to solve these proximity operators: one is dual and the other is primal-dual.
Following this way, we fortunately find that three existing iterative algorithms including Condat and Vu algorithm, primal-dual fixed point (PDFP) algorithm and primal-dual three operator (PD3O) algorithm are a special case of our proposed iterative algorithms. Moreover, we discover a new kind of iterative algorithm to solve the considered optimization problem, which is not covered by the existing ones. Under mild conditions, we prove the convergence of the proposed iterative algorithms. Numerical experiments applied on fused Lasso problem, constrained total variation regularization in computed tomography (CT) image reconstruction and low-rank total variation image super-resolution problem demonstrate the effectiveness and efficiency of the proposed iterative algorithms.

\end{abstract}




\newpage

\section{Introduction}
\label{sec:intro}  % \label{} allows reference to this section
Many problems in signal and image processing, machine learning and statistical prediction can be transformed into the minimization of the sum of two convex functions, a number of efficient iterative algorithms have been proposed in the last ten decades. Here we briefly review some general iterative algorithms. In particular, the forward-backward splitting algorithm \cite{combettes2005} is widely used to solve the following optimization problem,
\begin{equation}\label{sum-two}
\min_{x\in X}\ f(x)+g(x),
\end{equation}
where $X$ is a Hilbert space, $f\in \Gamma_{0}(X)$ is differentiable with a Lipschitz continuous gradient, $g\in \Gamma_{0}(X)$ is simple and maybe nonsmooth.  Here and in what follows, $\Gamma_{0}(X)$ denotes the class of all the proper lower-semicontinuous (lsc) convex functions from a Hilbert space $X$ to $(-\infty, +\infty]$.  The iterative shrinkage-thresholding algorithm (ISTA) \cite{daubechies2004} could be viewed as a special case of the forward-backward splitting algorithm applied to (\ref{sum-two}) when $f(x)= \frac{1}{2}\|Ax-b\|_{2}^{2}$ and $g(x)=\lambda \|x\|_1$, where $A\in R^{m\times n}$, $b\in R^{n}$ and $\lambda >0$. As a generalization of the ISTA, Beck and Teboulle \cite{beck2009} proposed a fast iterative shrinkage-thresholding algorithm (FISTA) on the basis of the Nesterov's accelerated gradient method \cite{nesterov1983}. The $\ell_1-$norm promotes sparsity for one-dimensional signal and plays an important role in the compressive sensing theory \cite{candes2006}. For two dimensional signal, such as image, the total variation (TV) regularization \cite{rof1992} has been widely used for image restoration and image reconstruction since its ability to maintain sharp edges of piecewise constant images(see, for example, \cite{chantf1999,bioucas2007,zhu2010,bonettini2012,Yongqingzhang2015TIP,zengty2016SCM}). Beck and Teboulle \cite{beckandteboulle2009} applied the FISTA to solve the total variation image deblurring problem. Since the proximity operator of the total variation has no closed-form solution, they applied the dual gradient method \cite{chambolle2004} to solve this sub-problem. It is observed that the total variation can be represented by a combination of a convex function with a first-order difference matrix, the following general optimization problem was gained much attention,
\begin{equation}\label{sum-two2}
\min_{x\in X}\ f(x)+h(Bx),
\end{equation}
where $X$ and $Y$ are two Hilbert spaces, $f\in \Gamma_{0}(X)$ and $h\in \Gamma_{0}(Y)$, $B:X\rightarrow Y$ is a linear transform. In the literature, the proximity operator of $f$ and $h$ are always assumed to be easily computed. Chambolle and Pock \cite{chambolleandpock2011} proposed a primal-dual proximity algorithm to solve (\ref{sum-two2}) and its' dual problem in the sense of Fenchel conjugate. The primal-dual proximity algorithm  belongs to the primal-dual method, it updates both a primal and dual variable in each iteration. Some other primal-dual methods to solve (\ref{sum-two2}) can be found in \cite{Esser2010,zhang2011}. It was pointed out in \cite{he2012} that the primal-dual proximity algorithm can be interpreted as a proximal point algorithm (PPA). Another important method to solve (\ref{sum-two2}) is the alternating direction of multiplier method (ADMM) \cite{boyd1}. The ADMM was shown to be equivalent to the Douglas-Rachford splitting method \cite{combettes2007} applied to the dual problem of (\ref{sum-two2}) in \cite{esser2009}. Compare the primal-dual proximity algorithm with ADMM, the former only needs the matrix-vector multiplications of operator $B$ and its adjoint operator, while ADMM needs to solve a non-trivial sub-problem. If $f$ is differentiable with a Lipschitz continuous gradient,
Chen et al. \cite{chenpj2013} proposed a primal-dual fixed point algorithm based on the proximity operator (PDFP$^2$O) to solve (\ref{sum-two2}).
When $f(x)=\frac{1}{2}\|Ax-b\|_{2}^2$, the PDFP$^2$O reduces to the generalized soft-thresholding algorithm proposed by Loris and Verhoeven \cite{loris2011}. The PDFP$^2$O could be viewed as a combination of the forward-backward splitting algorithm \cite{combettes2005} with the fixed point algorithm \cite{Micchelliandshen2011} for computing the proximity operator of $h \circ B$, which is also a primal-dual method. We refer the interested reader to \cite{komodakis2015IEEE} for a comprehensive review of the recent development of the primal-dual related algorithms.

In recent years, the minimization of the sum of three convex functions was received much attention that takes the form of
\begin{equation}\label{three-sum}
\min_{x\in X}\ f(x) + g(x) + h(Bx),
\end{equation}
where $X$ and $Y$ are two Hilbert spaces, $f\in \Gamma_{0}(X)$ is differentiable with a Lipschitz continuous gradient, $g\in \Gamma_{0}(X)$ and $h\in \Gamma_{0}(Y)$ maybe nonsmooth,
$B: X\rightarrow Y$ is a bounded linear operator. In the following, we always assume that the proximal function associated with $g$ and $h$ are easy to be computed. It is obvious that (\ref{sum-two}) and (\ref{sum-two2}) are a special case of (\ref{three-sum}).
Although the iterative algorithms mentioned before can be exploited to solve the considered three sum of convex functions problem (\ref{three-sum}), the corresponding iterative algorithm would either introduce more dual variables or doesn't make use of the differentiable property of the function $f(x)$.
 To overcome these drawbacks, there are several efficient iterative algorithms have been proposed to solve the optimization problem (\ref{three-sum}). Condat \cite{condat2013} proposed a primal-dual splitting algorithm to solve (\ref{three-sum}). Vu \cite{vu2013ACM} proposed an iterative algorithm to solve a zero of a monotone operator inclusion problem, which was based on the primal-dual splitting framework proposed in \cite{combettes2012}. It's observed that Vu's algorithm coincides with Condat's algorithm when applied to solve the particular optimization problem (\ref{three-sum}). In the following, we  use Condat-Vu algorithm to refer the iterative algorithm proposed by Condat \cite{condat2013} and Vu \cite{vu2013ACM} independently.
The Condat and Vu algorithm includes several existing iterative algorithms as its special case. When $f=0$, it becomes Chambolle-Pock algorithm \cite{chambolleandpock2011,pock1}, also known as the primal-dual hybrid gradient method (PDHG) \cite{zhu2008,Esser2010}, and proximal forward-backward splitting algorithm when $h=0$ \cite{combettes2005}. It has two iterative parameters restricted by the operator norm $B$ and the Lipschitz constant of $\nabla f$ together. Li and Zhang \cite{liqia2016} proved the convergence of a general primal-dual splitting algorithm, which included Condat-Vu algorithm as a special case. They proved the iteration schemes have $O(1/k)$ convergence rate both in the ergodic sense and in the sense of partial primal-dual gap. They also introduced a quasi-Newton and an overrelaxation strategies for accelerating the corresponding algorithms. Wen et al. \cite{wenmeng2016} proposed a self-adaptive primal-dual splitting algorithm, which didn't require to know the operator norm $B$. Latafat and Patrinos \cite{Latafat2017COP} proposed an asymmetric forward-backward-adjoint (AFBA) splitting algorithm, which included the Condat and Vu algorithm.
 As a generalization of the primal-dual fixed point algorithm based on proximity operator \cite{chenpj2013,loris2011} and the preconditioned alternating projection algorithm (PAPA) \cite{krol2012IP}, Chen et al. \cite{chenpj2016} proposed a so-called primal-dual fixed point (PDFP) algorithm to solve the optimization problem (\ref{three-sum}). They proved the convergence of the PDFP based on the fixed point theory and also obtain the convergence rate of the iteration scheme under suitable conditions. Very recently, Yan \cite{yanm2016} proposed a new primal-dual algorithm for solving the optimization problem (\ref{three-sum}), namely primal-dual three operator (PD3O). It reduces to the primal-dual proximity algorithm of Chambolle-Pock algorithm \cite{chambolleandpock2011,pock1} when $f=0$ and PDFP$^{2}$O \cite{chenpj2013} when $g=0$. In addition, it recovers the three-operator splitting scheme developed by Davis and Yin \cite{davis2015} when $B$ is the identity operator.

As far as we know, Condat-Vu algorithm \cite{condat2013,vu2013ACM}, PDFP \cite{chenpj2016} and PD3O \cite{yanm2016} are three main type of methods that designed to solve the sum of three convex functions (\ref{three-sum}). However, there are no results to show the relation of these iterative algorithms with other iterative algorithms. In this paper, we will present a unified framework to derive efficient iterative algorithms for solving the optimization problem (\ref{three-sum}), which are based on the forward-backward splitting method \cite{combettes2005} and the three operator splitting method \cite{davis2015}, respectively. Our main works are divided into two parts: (1)\ Based on the forward-backward splitting method, we need to compute the proximity operator $g + h\circ B$, which has no closed-form solution in general. To tackle this difficulty, we exploit two approaches: one is dual and the other is primal-dual. In the dual approach, we derive the dual of the minimization problem related to the proximity operator $g + h\circ B$. Thanks to the simple form of function $g(x)$, we show that the corresponding dual problem is the sum of two convex functions with one is differentiable with a Lipschitz continuous gradient and the other is the Fenchel conjugate of $h$. Then, the dual problem is solved once again by the forward-backward splitting algorithm and the primal optimal solution is obtained via the dual optimal solution. In the primal-dual approach, we solve the proximity minimization of $g(x)+h(Bx)$ based on the primal-dual proximity algorithm. We also point out the connection of our proposed iterative algorithms with PDFP \cite{chenpj2016} and Condat-Vu algorithm \cite{condat2013,vu2013ACM}, respectively. (2)\ We employ the three operator splitting method to solve the considered optimization problem (\ref{three-sum}). In the three operator iteration scheme, we need to compute the proximity operator of $h\circ B$. At this point, we also use the dual and primal-dual method to solve the proximity operator of $h\circ B$. The connection of the PD3O \cite{yanm2016} with our proposed iterative algorithm is presented. We show that PDFP \cite{chenpj2016} and PD3O \cite{yanm2016} are identical.  We also obtain a new iterative algorithm to solve the optimization problem (\ref{three-sum}) from the primal-dual approach.  To demonstrate the efficiency and effectiveness of the proposed iterative algorithms, we focus on applying them to solve the fused Lasso problem (\ref{fused-lasso}), the constrained total variation regularization problem (\ref{constrained-tv}) and the low-rank total variation image super-resolution problem (\ref{lrtv}). We present a detailed discussion on how to select the parameters for the proposed iterative algorithms based on the numerical results.

 \noindent \textbf{Fused Lasso problem.}\ The Fused Lasso problem was first introduced by Tibshirani et al. \cite{tibshirani-JRSSB-2005} which is an extension of the Lasso problem \cite{Tibshirani1996}. It is dedicated to finding a vector with sparsity in both of the coefficients and their successive differences. The optimization problem is
\begin{equation}\label{fused-lasso}
\min_{x\in R^n}\ \frac{1}{2}\|Ax-b\|_{2}^2 + \mu_1 \|x\|_1 + \mu_2 \|Dx\|_{1},
\end{equation}
where $\mu_1 >0$ and $\mu_2 >0$ are regularization parameters, $A\in R^{m\times n}$, $b\in R^{n}$ and the difference matrix $D_{(n-1)\times n}$ is defined by
$$
D = \left(
      \begin{array}{ccccc}
        -1 & 1 &  &  &  \\
         & -1 & 1 &  &  \\
         &  & \cdots &  &  \\
         &  &  & -1 & 1 \\
      \end{array}
    \right).
$$
The fused Lasso problem was applied in many fields, see for example \cite{tibshiraniandwang2008,liuj2010,yegb2011,lixx2014}. When $A=I$, the fused Lasso problem reduces to proximity operator of $\mu_1 \|x\|_1 + \mu_2 \|Dx\|_1$, which was first studied by Frideman et al. \cite{Friedman2007}.   Liu et al. \cite{liuj2010} simplified the method of Frideman et al. \cite{Friedman2007}. Recently, Shi et al. \cite{shihjm2017} considered the proximity operator of $g(x)+\mu_2 \|Dx\|_1$, where $g\in \Gamma_{0}(X)$, which generalized the results of Frideman et al. \cite{Friedman2007} and Liu et al. \cite{liuj2010}. Our considered problem (\ref{three-sum}) is more general than the results of \cite{Friedman2007,liuj2010,shihjm2017}.

\noindent \textbf{Constrained total variation regularization problem.}\ The total variation (TV) \cite{rof1992} regularization is widely used in the image restoration and image reconstruction problems \cite{sidky20083,hager2015JORSC}. To recover an image $x$ is often reduces to solve the following constrained TV regularization problem
\begin{equation}\label{constrained-tv}
\min_{x\in C}\ \frac{1}{2}\|Ax-b\|_{2}^2 + \mu \|x\|_{TV},
\end{equation}
Or equally
\begin{equation}
\min_{x\in R^n}\ \frac{1}{2}\|Ax-b\|_{2}^2 + \mu \|x\|_{TV}+ \delta_{C}(x),
\end{equation}
where $A\in R^{m\times n}$ is the system matrix, $b\in R^{m}$ is the collected data which is corrupted by noise, and $\mu >0$ is the regularization parameter which gives the balance between the data error term and the regularization term; $C$ is an nonempty closed convex set which is set according to the prior information on the image, such as nonnegative or bound range. $\delta_{C}$ denotes the indicator function of the closed convex set $C$ such that $\delta_{C}(x)=0$ if $x\in C$ and $+\infty$ otherwise. We note that the total variation $\|x\|_{TV}$ can be represented by a combination of convex function with a discrete gradient operator (See Definition \ref{tv-definition}).

\noindent \textbf{Low-rank total variation image super-resolution problem.}\ Besides the TV regularization, low-rank regularization is widely used in image restoration for recovering missing values in an image. A low-rank total variation (LRTV) image super-resolution optimization problem takes the form of
\begin{equation}\label{lrtv}
\min_{X\in R^{m\times n}}\ \frac{1}{2}\| DSX - T \|_{F}^{2} + \lambda_1 \|X\|_{*} + \lambda_2 \|X\|_{TV},
\end{equation}
where $\lambda_1 >0$ and $\lambda_2 >0$ are two regularization parameters, $T$ denotes the observed lower-resolution image, $D$ is a down-sampling operator, $S$ is a blurring operator and $X$ is the high-resolution image that we want to recover. Here, $\|\cdot\|_{F}$ represents the Frobenius-norm of matrix and $\|X\|_{*}$ denotes the nuclear norm of matrix $X$. The LRTV image super-resolution optimization problem (\ref{lrtv}) was first appeared in \cite{shifeng2013} and  further studied in \cite{shif2015}. It is a generalization of image super-resolution by TV-regularization \cite{marquina2008}.

The rest of the paper is organized as follows. In Section \ref{basic_notation}, we provide some basic definitions and lemmas from convex analysis. In Section \ref{forward-backward-dual-primal-dual}, we present the forward-backward splitting method combined with the dual and primal-dual approach to solve the considered optimization problem (\ref{three-sum}). We prove the convergence of the proposed iterative algorithms in finite-dimensional Hilbert spaces. We also point out the relation of some existing iterative algorithms with our iterative algorithms. In Section \ref{three-operator-dual-primal-dual}, we use the three operator splitting method together with the dual and primal-dual methods to solve (\ref{three-sum}). We show the convergence of the related iterative algorithms. The relationship between our iterative algorithms and the existing ones will also be included. In Section \ref{numer_test}, we apply the proposed iterative algorithms to solve the fused Lasso problem, the constrained total variation regularization problem and the low-rank total variation image super-resolution problem. Finally, we give some conclusions.

%%%%%%%%%%%%%%%%%%%%%%%%%%%%%%%%%%%%%%%%%%%%%%%%%%%%%%%%%%%%%%
\section{Preliminaries}
\label{basic_notation}
Throughout the paper, let $X$ be a Hilbert space equipped with the inner product $\langle \cdot, \cdot\rangle$
 and associated norm $\| \cdot \|$. Let $I$ be the identity operator on $X$. The operator norm of $B$ is denoted by
 $\|B\| = \sup_{x\in X}\frac{\|Bx\|}{\|x\|}$. The adjoint of an operator $B: X\rightarrow Y$ by $B^{*}$ such that $\langle x, By\rangle = \langle B^{*}x, y \rangle$.  We will denote $C\subseteq X$ as a nonempty, closed and convex set, the interior of $C$ will be denoted int $C$ and its relative interior ri $C$. Denote $\Gamma_{0}(X)$  as the set of all proper lower-semicontinuous (lsc) convex functions from $X$ to $(-\infty,+\infty]$.

 A real-valued function $f: X \rightarrow (-\infty,+\infty]$ is coercive, if $\lim_{\|x\|\rightarrow +\infty}f(x) = +\infty$. The domain of $f$ is defined by  dom $ f= \{ x\in X: f(x) < +\infty \} $  and $f$ is proper if dom $f\neq \emptyset$. We say that a real-valued function $f$ is lower semi-continuous (lsc) if $\lim\inf_{x\rightarrow x_0} f(x) \geq f(x_0)$. The subdifferential of $f$ is the set-valued operator $\partial f :X\rightarrow 2^{X} : x\mapsto \{ u\in X | f(y) \geq f(x) + \langle u, y-x\rangle, \forall y\in X \}$.

 The following fact will be required.

 \begin{lem}(\cite{Zalinescu2012})
 Let $g\in \Gamma_{0}(X)$, $h\in \Gamma_{0}(Y)$ and $B:X\rightarrow Y$ is a bounded linear operator such that $0\in \emph{int}(B (\emph{dom} g)  - \emph{dom} h )$, then $\partial (g + h\circ B) = \partial g + B^{*} \circ \partial h \circ B$.
 \end{lem}

We will use the concept of Fenchel Conjugate or Convex Conjugate to derive the dual of the considered optimization problem, which we recall below.

\begin{definition}(Fenchel Conjugate)
Let $f: X \rightarrow (-\infty, \infty ]$, the Fenchel Conjugate of $f$ denoted by $f^{*}$ is defined by
\begin{equation}
f^{*}(u) = \sup_{x} \langle x, u \rangle - f(x).
\end{equation}
\end{definition}
When $f$ is a proper lsc convex function, then $f^{**} = f$. For any $(x,u)\in X\times X$, $u\in \partial f(x)\Leftrightarrow x\in \partial f^{*}(u)$.

1962, Moreau \cite{Moreau1962} introduced the notion of proximity operator, which plays a very important role in designing proximal algorithms to solve convex optimization problems. A comprehensive review of the proximal algorithms can be found in \cite{parikh2014FTO}.
\begin{definition}(Proximity operator)
Let $f: X \rightarrow R \cup \{+\infty\}$ be a  proper lsc convex function. For $\lambda >0$, the proximity operator $prox_{\lambda f} : X \rightarrow X$ of $\lambda f$ is defined by
\begin{equation}
prox_{\lambda f}(v) = \arg\min_{x} \big\{ \frac{1}{2}\|x-v\|^{2} + \lambda f(x)   \big\}.
\end{equation}
\end{definition}

The proximity operator has many important properties. For instance, $prox_{\lambda f}$ is firmly nonexpansive, i.e.,
$$
\| prox_{\lambda f}(x) - prox_{\lambda f}(y)  \|^2 + \| (x-prox_{\lambda f}(x)) - (y-prox_{\lambda f}(y)) \|^2 \leq \|x-y\|^2, \ \forall x,y\in X.
$$
Equivalent,
$$
\|prox_{\lambda f}(x) - prox_{\lambda f}(y)\|^2 \leq \langle x-y, prox_{\lambda f}(x) - prox_{\lambda f}(y) \rangle, \forall x,y\in X,
$$
which means that the proximity operator is nonexpansive due to the Cauchy-Schwarz inequality.
The proximity operator is indeed an extension of the orthogonal projection operator. In fact, let $f(x) = \delta_{C}(x)$, then the proximity operator
 $prox_{\lambda \delta_{C}}(v) = P_{C}(v)$. The proximity operator of $\ell_1$-norm is the soft-thresholding operator which is commonly used in the sparse related optimization problems.
  For other interesting functions, we refer the readers to \cite{bauschkebook2011} for the explicit form of proximity operators.

\begin{definition}(Moreau envelope)
Let $f\in \Gamma_{0}(X)$ and $\lambda >0$, the Moreau envelope or Moreau-Yosida regularization is given by
\begin{equation}
\widetilde{f}_{\lambda}(x) = \inf_{y} \{ f(y) + \frac{1}{2\lambda}\|x-y\|^2  \}.
\end{equation}

\end{definition}

The next lemma shows that the Moreau-Yosida regularization is differentiable on $X$.
\begin{lem}(\cite{bauschkebook2011})\label{lem-differe-moreau-yosida}
Let $\lambda >0$ and $f\in \Gamma_{0}(X)$, then the function $f_{\lambda}(X): X \rightarrow R$ is differentiable and its gradient
\begin{equation}
\nabla \widetilde{f}_{\lambda}(x) = \frac{1}{\lambda}(x-prox_{\lambda f}(x)),
\end{equation}
with $1/\lambda$-Lipschitz continuous.
\end{lem}

The Moreau equality presents the computation of proximity operator a  proper lsc convex function $f$ from its convex conjugate $f^{*}$ or converse.
\begin{lem}(\cite{bauschkebook2011})\label{lem-moreau-equality}
For any $\lambda >0$ and a vector $u$, the Moreau equality takes the form of
\begin{equation}
prox_{\lambda f}(u) + \lambda prox_{\frac{1}{\lambda}f^{*}}(\frac{1}{\lambda}u) = u.
\end{equation}
\end{lem}

By the definition of convex conjugate and proximity operator, it is easy to prove the following result.
\begin{lem}\label{lem-proximity-conjugate}
Let $\lambda >0$ and $f\in \Gamma_{0}(X)$, for any $u\in X$, then
\begin{equation}
prox_{(\lambda f)^{*}}(u) = \lambda prox_{\frac{1}{\lambda}f^{*}}(\frac{1}{\lambda}u).
\end{equation}
\end{lem}

\begin{proof}
By the definition of convex conjugate, for any $u\in X$, we have
\begin{align}\label{eq2-7}
(\lambda f)^{*}(u) & = \sup_{x} \langle x,u \rangle - (\lambda f)(x), \nonumber \\
& = \lambda \sup_{x} \langle x, \frac{1}{\lambda}u \rangle - f(x), \nonumber  \\
& = \lambda f^{*}(\frac{1}{\lambda}u).
\end{align}
It follows from the above relation and also together with the definition of proximity operator, we get
\begin{align}
prox_{(\lambda f)^{*}}(u) & = \arg\min_{x}\ \{  \frac{1}{2}\| x-u  \|^2 + (\lambda f)^{*}(x)   \}, \nonumber \\
& = \arg\min_{x}\ \{ \frac{1}{2}\|x-u\|^2 + \lambda f^{*}(\frac{1}{\lambda}x)   \}.
\end{align}
Let $x' = \frac{1}{\lambda}x$, then
\begin{align}
\quad & \min_{x}\ \frac{1}{2}\|x-u\|^2 + \lambda f^{*}(\frac{1}{\lambda}x), \nonumber \\
= & \min_{x'}\ \frac{1}{2}\| \lambda x' -u \|^2 + \lambda f^{*}(x') \nonumber \\
= & \min_{x'}\ \lambda^2 ( \frac{1}{2}\|x' - \frac{1}{\lambda}u\|^2 + \frac{1}{\lambda}f^{*}(x')  ).
\end{align}
Therefore, $prox_{(\lambda f)^{*}}(u) = \lambda prox_{\frac{1}{\lambda}f^{*}}(\frac{1}{\lambda}u)$.
\end{proof}

We shall make full use of the following lemmas to derive our iterative algorithm. The lemma was proved in \cite{combettesSVA2010}. See also \cite{bauschkebook2011}.

\begin{lem}(\cite{combettesSVA2010})\label{key-lemma}
Let $g\in \Gamma_{0}(X)$ and $h\in \Gamma_{0}(Y)$. Let $u\in X$ and $r\in Y$. Let $B: X\rightarrow Y$ be a bounded linear operator satisfying the following condition:$r\in \emph{int} (B(\emph{dom}\  g) - \emph{dom}\  h)$.
Consider the minimization problem of
\begin{equation}\label{lem-three-sum}
\min_{x\in X}\ \frac{1}{2}\| x- u \|^{2} +  g(x) +  h(Bx-r),
\end{equation}
Then the following hold:

\noindent \emph{(i)} The dual problem of (\ref{lem-three-sum}) is
\begin{equation}\label{lem-dual-problem}
\begin{aligned}
\max_{y\in Y}\ & - \frac{1}{2}\|  B^{*}y - u \|^2 + \widetilde{g}_{1}(u-B^{*}y) -  h^{*}(y) - \langle y, r\rangle + \frac{1}{2}\|u\|^2.
\end{aligned}
\end{equation}

\noindent \emph{(ii)} Let $y^{*}$ is an optimal solution of the dual problem (\ref{lem-dual-problem}) and $x^{*} = prox_{ g}(u -  B^{*}y^{*})$. Then, $x^{*}$ is an optimal solution of the primal problem (\ref{lem-three-sum}).

\noindent \emph{(iii)} The optimal value of the primal problem (\ref{lem-three-sum}) is equal to the optimal value of the dual problem (\ref{lem-dual-problem}).

\end{lem}

The following lemma can be found in \cite{liqia2016}.

\begin{lem}(\cite{liqia2016})\label{key-lemmalemma}
Let $f_1 \in \Gamma_{0}(X)$ and $f_2 \in \Gamma_{0}(Y)$. Let $B:X\rightarrow Y$ be a bounded linear operator such that $0\in \emph{int} (B(\emph{dom}\  f_1) - \emph{dom}\  f_2)$. Consider the following general optimization problem
\begin{equation}\label{key-lemma2}
\min_{x}\ f_{1}(x) + f_2 (Bx).
\end{equation}
Let $x$ is a solution of (\ref{key-lemma2}), then for any $\sigma>0, \tau >0$, there exists a vector $y\in Y$ such that
\begin{equation}
\begin{aligned}
x & = prox_{\tau f_1} (x- \tau B^{*}y),\\
y & = prox_{\sigma f_{2}^{*}} ( y + \sigma Bx).
\end{aligned}
\end{equation}
Conversely, if there exists $\sigma>0, \tau >0$, $x\in X$ and $y\in Y$ satisfying (\ref{key-lemma2}), then $x$ is a solution of (\ref{key-lemma2}).
\end{lem}

%%%%%%%%%%%%%%%%%%%%%%%%%%%%%%%%%%%%%%%%%%%%%%%%%%%%%%%%%%%%%
\section{A forward-backward splitting method to solve (\ref{three-sum})}
\label{forward-backward-dual-primal-dual}

In this section, we will propose efficient iterative algorithms to solve the optimization problem (\ref{three-sum}). We make the following two assumptions throughout the paper: (1) The minimizers of the optimization problem (\ref{three-sum}) is always exists; (2) $0 \in \emph{int} (B(\emph{dom}\ g) - \emph{dom}\ h)$.
As we have mentioned in the introduction, since the function $f(x)$ in the optimization problem (\ref{three-sum}) is smooth, we can employ the forward-backward splitting algorithm with errors \cite{combettes2005} and obtain the iteration scheme as follows. For any $x^{0}\in X$,
\begin{equation}\label{main-forward-backward}
x^{k+1} = prox_{\gamma (g + h \circ B)}(x^k - \gamma \nabla f(x^{k})) + e_k,
\end{equation}
where $\gamma \in (0, 2/L)$, $L$ is the Lipschitz constant of $\nabla f$, $e_k$ denotes the error between $x^{k+1}$ and the proximity operator $prox_{\gamma (g + h \circ B)}(x^k - \gamma \nabla f(x^{k}))$. The following convergence theorem was proved in \cite{combettes2005}.
\begin{thm}(\cite{combettes2005})\label{main-convergence-theorem}
Let $\gamma \in (0, 2/L)$, $L$ is the Lipschitz constant of $\nabla f$. For any $x^{0}\in X$,  the iterative sequence $\{x^{k}\}$ is defined by (\ref{main-forward-backward}). We assume that
 $\sum_{n=0}^{\infty}\|e_k\| < +\infty$. Then the iterative sequence $\{x^k\}$ converges weakly to a solution of the optimization problem (\ref{three-sum}).
\end{thm}

Since the proximity operator of the function $\gamma (g+h \circ B)$ has no closed-form solution, so we cannot obtain the exact value of $x^{k+1}$ with the error vector $e_k =0$. However, if we can obtain a sufficient approximation of the proximity operator so that the error vector $e_k$ satisfies the requirement of Theorem \ref{main-forward-backward}, the obtained iterative sequences will still converge to the solution of the optimization problem (\ref{three-sum}).
To get the updated sequences $\{x^{k+1}\}$, the key problem is to efficient compute the proximity operator of $\gamma (g + h \circ B)$, i.e., $prox_{\gamma (g + h \circ B)}$. In the following, we consider two efficient approaches to solving the proximity operator of function $\gamma (g + h \circ B)$ at the point $x^k - \gamma \nabla f(x^{k})$: one is from dual and the other is from primal-dual. Both of the two methods can produce an approximation solution of the proximity operator $prox_{\gamma (g + h \circ B)}(x^k - \gamma \nabla f(x^{k}))$  and the error will satisfy the requirement of Theorem \ref{main-convergence-theorem}.

\subsection{Dual approach}

First, we consider the dual approach. Recall that the proximity operator of $prox_{\gamma (g + h \circ B)}(x^k - \gamma \nabla f(x^{k}))$ is the minimizer of the following optimization problem
\begin{equation}\label{eq3-1}
\min_{x\in X}\ \left\{\frac{1}{2}\| x - (x^k - \gamma \nabla f(x^{k})) \|^2  + \gamma g(x) + \gamma h(Bx)\right\},
\end{equation}
which is a special case of (\ref{lem-three-sum}).
In Lemma \ref{key-lemma}, let $r=0$, $u = x^k - \gamma \nabla f(x^{k})$. Define $g := \gamma g$ and $h := \gamma h$. Then we obtain the dual formulation of the minimization problem (\ref{eq3-1}) is
\begin{align}\label{eq3-2}
& \max_{y\in Y}\  -\frac{1}{2} \|B^{*}y - u\|^2 + \gamma \widetilde{g}_{\gamma}(u-B^{*}y) -  (\gamma h)^{*}(y) + \frac{1}{2}\|u\|^2 \nonumber \\
& = \max_{y\in Y}\ -\frac{1}{2} \|B^{*}y - u\|^2 + \gamma \widetilde{g}_{\gamma}(u-B^{*}y) -  \gamma h^{*}(\frac{1}{\gamma}y) + \frac{1}{2}\|u\|^2.
\end{align}
The equality in (\ref{eq3-2}) is due to the fact that $(\gamma h)^{*}(y) = \gamma h^{*}(\frac{1}{\gamma}y)$. Let $y$ is a solution of the dual problem (\ref{eq3-2}), by Lemma \ref{key-lemma} (ii), we know that $x=prox_{\gamma g}(u-B^{*}y)$ is the optimal solution of (\ref{eq3-1}). In (\ref{eq3-2}), let $y' = \frac{1}{\gamma}y$, then the dual problem (\ref{eq3-2}) reduces to
 \begin{equation}\label{eq3-2-1}
 \max_{y'\in Y}\  \gamma^2 (-\frac{1}{2} \|B^{*}y' - \frac{1}{\gamma}u\|^2 + \frac{1}{\gamma} \widetilde{g}_{\gamma}(u- \gamma B^{*}y') -  \frac{1}{\gamma} h^{*}(y') + \frac{1}{2\gamma^2}\|u\|^2 ).
\end{equation}
The corresponding optimal solution $x$ of (\ref{eq3-1}) becomes $x=prox_{\gamma g}(u- \gamma B^{*}y')$.
Let $F(y') = \frac{1}{2} \|B^{*}y' - \frac{1}{\gamma}u\|^2 - \frac{1}{\gamma} \widetilde{g}_{\gamma}(u- \gamma B^{*}y') $, then the optimal solution of (\ref{eq3-2-1}) is equal to the minimizer of the following minimization problem
\begin{equation}\label{eq3-3}
\min_{y'\in Y}\ F(y') + \frac{1}{\gamma}h^{*}(y').
\end{equation}
By Lemma \ref{lem-differe-moreau-yosida}, we have $\nabla F(y') = -\frac{1}{\gamma}B prox_{\gamma g}(u - \gamma B^{*}y')$ and $\nabla F(y')$ is Lipschitz continuous with Lipschitz constant $\lambda_{max}(BB^{*})$, where $\lambda_{max}(BB^{*})$ denotes the spectral radius of operator $BB^{*}$. In fact, since the proximity operator is nonexpansive, then for any $y_1, y_2\in Y$, we have
\begin{align}
\| \nabla F(y_1) - \nabla F(y_2)    \| & = \| \frac{1}{\gamma}Bprox_{\gamma g}(u - \gamma B^{*}y_1) - \frac{1}{\gamma}B prox_{\gamma g}(u - \gamma B^{*}y_2)   \| \nonumber \\
& \leq \frac{1}{\gamma} \|B\| \| prox_{\gamma g}(u - \gamma B^{*}y_1) - prox_{\gamma g}(u - \gamma B^{*}y_2)   \| \nonumber \\
& \leq \frac{1}{\gamma}\|B\| \|  (\gamma B^{*}y_1 - \gamma B^{*}y_2) \| \nonumber \\
& \leq \lambda_{max}(BB^{*}) \| y_1 - y_2 \|.
\end{align}
Therefore, we obtain the following iteration scheme to solve (\ref{eq3-3}), which is based on the forward-backward splitting algorithm. For any $y^{0}\in Y$, choose $0 < \lambda < 2/\lambda_{max}(B^{*}B)$,
\begin{equation}
y^{j_k +1} = prox_{\frac{\lambda}{\gamma}h^{*}} (y^{j_k}  + \frac{\lambda}{\gamma}B prox_{\gamma g}(u - \gamma B^{*}y^{j_k}) ),\ j_k = 0, 1, 2, \cdots.
\end{equation}
 In conclusion, we propose the following iterative algorithm  to solve the optimization problem (\ref{three-sum}). The iterative algorithm includes an outer iteration step and an inner iteration step, respectively.

\begin{algorithm}[H]
\caption{A dual forward-backward splitting  algorithm for solving the optimization problem (\ref{three-sum})}
\begin{algorithmic}\label{dual-forward-backward}
\STATE \textbf{Initialize:}  Given arbitrary $x^{0}\in X$ and $y^{0}\in Y$. Choose $\gamma \in (0,2/L)$ and $\lambda \in (0 , 2/\lambda_{max}(BB^{*}))$.
\STATE 1. (Outer iteration step)\ For $k=0, 1, 2, \cdots$
\STATE \quad $u^{k} = x^k - \gamma \nabla f(x^k)$;
\STATE 2. (Inner iteration step)\ For $j_k = 0, 1, 2, \cdots$
\STATE \quad $y^{j_k +1} = prox_{\frac{\lambda}{\gamma}h^{*}} (y^{j_k}  + \frac{\lambda}{\gamma}B prox_{\gamma g}(u^{k} - \gamma B^{*}y^{j_k}) )$;
\STATE  End inner iteration step when the stopping criteria reached. Output: $y^{J_k}$.
\STATE 3. Update $x^{k+1} = prox_{\gamma g}(u^k - \gamma B^{*}y^{J_k})$;
\STATE 4. End the outer iteration step when some stopping criteria reached.
\end{algorithmic}
\end{algorithm}

We prove the convergence of Algorithm \ref{dual-forward-backward} in finite-dimensional Hilbert spaces.

\begin{thm}\label{convergence-dual-forward-backward}
Let $\gamma \in (0,2/L)$ and $\lambda \in (0 , 2/\lambda_{max}(BB^{*}))$. For any $x^{0}\in X$ and $y^{0}\in Y$, the iterative sequence $\{x^k\}$ and $\{y^k\}$ are generated by Algorithm \ref{dual-forward-backward}. Then  the iterative sequence $\{x^k\}$ converges to a solution of the optimization problem (\ref{three-sum}).

\end{thm}

\begin{proof}
Let $y$ is the optimal solution of the minimization problem (\ref{eq3-3}), it follows from Lemma \ref{key-lemma} that $x = prox_{\gamma g}(u^k - \gamma B^{*}y)$ is the optimal solution of (\ref{eq3-1}). According to the classical convergence of the forward-backward splitting algorithm, we have $\|y^k -y\| \rightarrow 0$ as $k\rightarrow \infty$. Therefore, for a given constant $1/(\gamma \|B\| k^2) >0$, there exists an integer $j_k$, when $J_k \geq j_k $, we have $\|y^{J_k}-y\| \leq 1/(\gamma \|B\| k^2)$. Notice that $x^{k+1} = prox_{\gamma g}(u^k - \gamma B^{*}y^{J_k})$, we get
\begin{align}
\|e_k\| &= \|x^{k+1} - x\| \nonumber \\
& = \| prox_{\gamma g}(u^k - \gamma B^{*}y^{J_k}) - prox_{\gamma g}(u^k - \gamma B^{*}y) \| \nonumber \\
& \leq \| \gamma B^{*}y^{J_k} - \gamma B^{*}y \| \leq \frac{1}{k^2}.
\end{align}
The first inequality comes from the fact that the proximity operator is nonexpansive and the second inequality is due to the estimation between $y^{J_k}$ and $y$. Then $\sum_{n=0}^{\infty}\|e_k\| < +\infty$. By Theorem \ref{main-convergence-theorem}, we can conclude that the iterative sequence $\{x^k\}$ converges to a solution of the optimization problem (\ref{three-sum}). This completes the proof.

\end{proof}

\subsection{Primal-dual approach}

In this part, we present a primal-dual method to solve the minimization problem (\ref{eq3-1}). We will see that the primal-dual method can also obtain an accurate approximation to the optimal solution of the minimization problem (\ref{eq3-1}).
We employ the primal-dual proximity algorithm developed in \cite{chambolleandpock2011} to solve (\ref{eq3-1}) and obtain the following iteration schemes.
For any $\overline{x}^{0}\in X$ and $y^{0}\in Y$, choose $\tau >0$ and $\sigma >0$ satisfying $\tau \sigma \|B\|^2<1$, let $j_k = 0, 1, 2, \cdots$, do
\begin{align}
\overline{x}^{j_k +1} & = prox_{\tau(  \frac{1}{2}\|\cdot - u \|^2 + \gamma g )}(\overline{x}^{j_k} - \tau B^{*}y^{j_k}), \\
y^{j_k +1} & = prox_{\sigma(\gamma h)^{*}} (y^{j_k} + \sigma B (2\overline{x}^{j_k +1} - \overline{x}^{j_k}) ).
\end{align}
After simple calculation, we obtain
\begin{align}
\overline{x}^{j_k +1} & = prox_{\frac{\tau \gamma}{1+\tau}g}(\frac{\overline{x}^{j_k} - \tau B^{*}y^{j_k} + \tau u}{1+\tau}), \\
y^{j_k +1} & = \gamma prox_{\frac{\sigma}{\gamma}h^{*}} (\frac{1}{\gamma}(y^{j_k} + \sigma B (2\overline{x}^{j_k +1} - \overline{x}^{j_k})) ).
\end{align}
Therefore, the detailed iterative algorithm is summarized in Algorithm \ref{primal-dual-forward-backward}.

\begin{algorithm}[H]
\caption{A primal-dual forward-backward splitting algorithm for solving the optimization problem (\ref{three-sum})}
\begin{algorithmic}\label{primal-dual-forward-backward}
\STATE \textbf{Initialize:}  Given arbitrary $x^{0}, \overline{x}^{0}\in X$ and $y^{0}\in Y$. Choose $\gamma \in (0,2/L)$. Let $\sigma>0$ and $\tau >0$ satisfy the
condition $\tau \sigma < \frac{1}{\|B\|^2}$.
\STATE 1. (Outer iteration step)\ For $k=0, 1, 2, \cdots$
\STATE \quad $u^{k} = x^k - \gamma \nabla f(x^k)$;
\STATE 2. (Inner iteration step)\ For $j_k = 0, 1, 2, \cdots$
\STATE \quad 2.a. $\overline{x}^{j_k +1}  = prox_{\frac{\tau \gamma}{1+\tau}g}(\frac{\overline{x}^{j_k} - \tau B^{*}y^{j_k} + \tau u^k}{1+\tau})$;
\STATE \quad 2.b. $y^{j_{k}+1} = \gamma prox_{\frac{\sigma}{\gamma}h^{*}} (\frac{1}{\gamma}(y^{j_k} + \sigma B (2\overline{x}^{j_k +1} - \overline{x}^{j_k})) )$;
\STATE \ End inner iteration when the stopping criteria reached and output $\overline{x}^{J_k}$
\STATE 3. Update $x^{k+1} = \overline{x}^{J_k}$.
\STATE 4. End the outer iteration step when some stopping criteria reached
\end{algorithmic}
\end{algorithm}

We prove the convergence of Algorithm \ref{primal-dual-forward-backward} in finite-dimensional Hilbert spaces.
\begin{thm}\label{convergence-primal-dual-forward-backward}
Let $\gamma \in (0,2/L)$. Let $\sigma>0$ and $\tau >0$ satisfy the
condition $\tau \sigma < \frac{1}{\|B\|^2}$. For any $x^{0}, \overline{x}^{0}\in X$ and $y^{0}\in Y$, the iterative sequences $\{x^k\}$, $\{\overline{x}^{k}\}$ and $\{y^k\}$ are generated by Algorithm \ref{primal-dual-forward-backward}.  Then the iterative sequence $\{x^k\}$ converges to a solution of the optimization problem (\ref{three-sum}).

\end{thm}

\begin{proof}
In Lemma \ref{key-lemmalemma}, let $f_{1}(x)=\frac{1}{2}\|x-u^k\|^2 + \gamma g(x)$, $f_2(x)=\gamma h(x)$. Let $x$ is an optimal solution of (\ref{eq3-1}), by Lemma \ref{key-lemmalemma}, for any $\tau >0$ and $\sigma >0$, there exists a vector $y$ such that
\begin{equation}
\begin{aligned}\label{eq3-3-3}
x & = prox_{\tau f_1} (x - \tau B^{*}y), \\
y & = prox_{\sigma f_2} (y + \sigma By).
\end{aligned}
\end{equation}
It follows from the definition of the proximity operator, the equation (\ref{eq3-3-3}) reduces to
\begin{equation}
\begin{aligned}
x & = prox_{\frac{\gamma \tau}{1+\tau}g} (\frac{x - \tau B^{*}y + \tau u^k}{1+\tau}),\\
y & = prox_{\sigma (\gamma h)^*} (y + \sigma By).
\end{aligned}
\end{equation}
Due to Theorem 1 of \cite{chambolleandpock2011}, we know that $\{\overline{x}^{k}\}$ and $\{y^k\}$ converge to $x$ and $y$, respectively. Therefore, for a given constant $\frac{1+\tau}{k^2}>0$, there exists an integer $j_{k}$, when $J_k \geq j_{k}$, we have $\|\overline{x}^{J_k}-x\|+\tau \|B\| \| y^{J_k}-y \| \leq (1+\tau)\frac{1}{k^2}$. Then, we obtain
\begin{align}
\| e^k \| = \| x^{k+1} - x  \| & = \| prox_{\frac{\gamma \tau}{1+\tau}g}(\frac{\overline{x}^{J_k}-\tau B^{*}y^{J_k} +\tau u^k }{1+\tau}) - prox_{\frac{\gamma \tau}{1+\tau}g}(\frac{x-\tau B^{*}y +\tau u^k }{1+\tau})    \| \nonumber \\
& \leq \| \frac{\overline{x}^{J_k}-\tau B^{*}y^{J_k} +\tau u^k }{1+\tau} - \frac{x-\tau B^{*}y +\tau u^k }{1+\tau}   \| \nonumber \\
& = \frac{1}{1+\tau}\| (\overline{x}^{J_k} - x)  - \tau B^{*}(y^{J_k} -y ))  \| \nonumber \\
& \leq \frac{1}{1+\tau} (\|\overline{x}^{J_k} - x\| + \tau \|B\| \|y^{J_k} -y\|)  \leq \frac{1}{k^2}.
\end{align}
Therefore, $\sum_{n=0}^{\infty}\| e^k \| < +\infty$. By Theorem \ref{main-convergence-theorem}, we can conclude that the iterative sequence $\{x^k\}$ converges to a solution of the optimization problem (\ref{three-sum}). This completes the proof.

\end{proof}

\subsection{Connections to other existing iterative algorithms}

In this subsection, we present the connections of the proposed iterative algorithm to some existing iterative algorithms.

In Algorithm \ref{dual-forward-backward}, let $j_k = k$ and the number of inner iteration equals to one, then it is reduced to the PDFP \cite{chenpj2016},
\begin{equation}\label{alg-PDFP}
 \left\{
\begin{aligned}
v^{k+1} & =  prox_{\gamma g}(x^k - \gamma \nabla f(x^k) - \gamma B^{*}y^{k}), \\
y^{k+1} & =  prox_{\frac{\lambda}{\gamma}h^{*}}(y^k + \frac{\lambda}{\gamma}Bv^{k+1}), \\
x^{k+1} & =  prox_{\gamma g}(x^k - \gamma \nabla f(x^k) - \gamma B^{*}y^{k+1}),
\end{aligned}
\right.
\end{equation}
With the help of Moreau equality (Lemma \ref{lem-moreau-equality}), the updated sequence $y^{k+1}$ is equal to
\begin{equation}
\begin{aligned}
y^{k+1} & =  prox_{\frac{\lambda}{\gamma}h^{*}}( \frac{\lambda}{\gamma}( \frac{\gamma}{\lambda} y^k + Bv^{k+1})), \\
& \quad = \frac{\lambda}{\gamma} (I - prox_{\frac{\gamma}{\lambda}h})(\frac{\gamma}{\lambda} y^k + Bv^{k+1}).
\end{aligned}
\end{equation}
Let $\overline{y}^{k} = \frac{\gamma}{\lambda}y^{k}$, then the PDFP iteration scheme (\ref{alg-PDFP}) can be rewritten as
\begin{equation}\label{alg-PDFP2}
 \left\{
\begin{aligned}
v^{k+1} & =  prox_{\gamma g}(x^k - \gamma \nabla f(x^k) - \lambda B^{*}\overline{y}^{k}), \\
\overline{y}^{k+1} & =  (I -prox_{\frac{\gamma}{\lambda}h})(y^k + Bv^{k+1}), \\
x^{k+1} & =  prox_{\gamma g}(x^k - \gamma \nabla f(x^k) - \lambda B^{*}\overline{y}^{k+1}),
\end{aligned}
\right.
\end{equation}
Chen et al. \cite{chenpj2016} proved the convergence of (\ref{alg-PDFP}) under the conditions that $0< \lambda < 1/\lambda_{max}(BB^{*})$ and $0< \gamma < 2/L$. Our proposed Algorithm \ref{dual-forward-backward} provides a more wide selection of the iterative parameter $\lambda$ than the PDFP \cite{chenpj2016}. It's observed that the PDFP \cite{chenpj2016} coincides with the dual forward-backward algorithm proposed in Combettes et al. \cite{combettesSVA2010} when $f(x)=\frac{1}{2}\|x-u\|_{2}^{2}$ and the iterative parameter $\gamma =1$.

In the next, we show the connection between Algorithm \ref{primal-dual-forward-backward} and the Condat-Vu algorithm.
In Algorithm \ref{primal-dual-forward-backward}, let $\overline{x}^{0}=x^{0}$, $j_k =k$ and fix the number of inner iteration with one, then the iteration scheme of Algorithm \ref{primal-dual-forward-backward} reduces to
\begin{equation}\label{alg-condat}
 \left\{
\begin{aligned}
x^{k+1} & =  prox_{\frac{\tau \gamma}{1+\tau} g}(\frac{x^k - \tau B^{*}y^k + \tau (x^k - \gamma \nabla f(x^k))}{1+\tau}), \\
y^{k+1} & =  \gamma prox_{\frac{\sigma}{\gamma}h^{*}}( \frac{1}{\gamma}y^k + \frac{\sigma}{\gamma}B(2x^{k+1} -x^k ) ).
\end{aligned}
\right.
\end{equation}
Let $\overline{y}^k = \frac{1}{\gamma}y^k$ and after simple calculation, then the iteration scheme (\ref{alg-condat}) can be represented as
\begin{equation}\label{alg-condat2}
 \left\{
\begin{aligned}
x^{k+1} & =  prox_{\frac{\tau \gamma}{1+\tau} g}(x^k - \frac{\tau\gamma}{1+\tau} B^{*}\overline{y}^k - \frac{\tau\gamma}{1+\tau} \nabla f(x^k)), \\
\overline{y}^{k+1} & =   prox_{\frac{\sigma}{\gamma}h^{*}}( \overline{y}^k + \frac{\sigma}{\gamma}B(2x^{k+1} -x^k ) ).
\end{aligned}
\right.
\end{equation}
Let $\sigma' = \frac{\sigma}{\gamma}$ and $\tau' = \frac{\tau \gamma}{1+\tau}$, then the above iteration scheme (\ref{alg-condat2}) recovers the primal-dual splitting algorithm proposed in \cite{condat2013},
\begin{equation}\label{alg-condat3}
 \left\{
\begin{aligned}
x^{k+1} & =  prox_{\tau' g}(x^k - \tau' B^{*}\overline{y}^k - \tau' \nabla f(x^k)), \\
\overline{y}^{k+1} & =   prox_{\sigma' h^{*}}( \overline{y}^k + \sigma' B(2x^{k+1} -x^k ) ).
\end{aligned}
\right.
\end{equation}
Condat \cite{condat2013} proved the iteration scheme (\ref{alg-condat3}) converges to a solution of the optimization problem (\ref{three-sum}) under the condition $\frac{1}{\tau'} - \sigma' \|B\|^2 > L/2$. This condition is satisfied with the choice of iterative parameters in Algorithm \ref{primal-dual-forward-backward}. In fact, since $\sigma = \gamma \sigma'$ and $\tau = \frac{\tau'}{\gamma -\tau'}$, it follows from the conditions of $\sigma \tau \|B\|^2 <1$ and $0<\gamma <2/L$, we have
\begin{align}
& \gamma \sigma' \frac{\tau'}{\gamma - \tau'}\|B\|^2 <1, \nonumber \\
& \Leftrightarrow \frac{1}{\tau'} - \sigma' \|B\|^2 > 1/\gamma > L/2.
\end{align}
In the above inequality, all the iterative parameters are mixed. While in our proposed Algorithm \ref{primal-dual-forward-backward}, we have more freedom to choose the iterative parameters than the Condat-Vu algorithm. We note that Yan \cite{yanm2016} presented a formulation of Condat-Vu algorithm as follows,
\begin{equation}\label{alg-condat4}
 \left\{
\begin{aligned}
x^{k+1} & =  prox_{\gamma g}(x^k - \gamma B^{*}\overline{y}^k - \gamma \nabla f(x^k)), \\
\overline{y}^{k+1} & =   prox_{\frac{\sigma}{\gamma}h^{*}}( \overline{y}^k + \frac{\sigma}{\gamma}B(2x^{k+1} -x^k ) ).
\end{aligned}
\right.
\end{equation}
Yan \cite{yanm2016} pointed out the iteration scheme (\ref{alg-condat4}) may diverges with the condition of $0<\gamma <2/L$ and $0<\sigma<1/\|B\|^2$. Under the same condition of the iterative parameters $\gamma$ and $\sigma$, we give another formulation of the Condat-Vu algorithm below, which is converged. In fact, let $\tau = 1$ in the iteration scheme (\ref{alg-condat2}), it reduces to
\begin{equation}\label{alg-condat5}
 \left\{
\begin{aligned}
x^{k+1} & =  prox_{\frac{ \gamma}{2} g}(x^k - \frac{\gamma}{2} B^{*}\overline{y}^k - \frac{\gamma}{2} \nabla f(x^k)), \\
\overline{y}^{k+1} & =   prox_{\frac{\sigma}{\gamma}h^{*}}( \overline{y}^k + \frac{\sigma}{\gamma}B(2x^{k+1} -x^k ) ).
\end{aligned}
\right.
\end{equation}
The difference between (\ref{alg-condat4}) and (\ref{alg-condat5}) is the iterative parameter $\gamma$ involved in the calculation of $x^{k+1}$. Under the choice of $\gamma$ belongs to $(0,2/L)$, the iteration scheme (\ref{alg-condat5}) maintains converge while the iteration scheme (\ref{alg-condat4}) may diverges as stated by Yan \cite{yanm2016}.

\section{A three operator splitting method to solve (\ref{three-sum})}
\label{three-operator-dual-primal-dual}

Recently, Yan \cite{yanm2016} proposed a primal-dual three operator (PD3O) splitting algorithm to solve the optimization problem (\ref{three-sum}).  In particular, when $B=I$, the  PD3O \cite{yanm2016} coincides with the three operator splitting algorithm proposed by Davis and Yin \cite{davis2015}. The three operator splitting method is a generalization of many well-known operators splitting methods, such as Forward-Backward splitting method \cite{passty1979JMAA,Chenhg1997}, Douglas-Rachford splitting method \cite{lionsandmercier1979,Eckstein1992}, Peaceman-Rachford splitting method \cite{peacemanandrachford1955,hebs2014SJO} and Forward-Douglas-Rachford splitting method \cite{briceno2015Optim}.
In this section, we will apply the three operator splitting algorithm to solve the optimization problem (\ref{three-sum}). The iteration scheme of the  three operator splitting algorithm with errors when applying to the optimization problem (\ref{three-sum}) takes the form of, for any $z^0 \in X$,
\begin{equation}\label{eq4-1}
\begin{aligned}
x^{k} & = prox_{\gamma g}(z^k), \\
s^k & = prox_{\gamma (h \circ B)}(2x^k - z^k - \gamma \nabla f(x^k)) +e_k, \\
z^{k+1} & = z^k + s^k -x^k,
\end{aligned}
\end{equation}
where $\gamma \in (0,2/L)$, $e_k$ is an error vector. We take the following convergence results of the three operator splitting algorithm (\ref{eq4-1}) from \cite{davis2015}.

\begin{thm}(\cite{davis2015})\label{main-convergence-theorem2}
Let $\gamma \in (0, 2/L)$, $L$ is the Lipschitz constant of $\nabla f$. For any $x^{0}\in X$,  the iterative sequence $\{x^{k}\}, \{s^k\}$ and $\{z^k\}$ are defined by (\ref{eq4-1}). We assume that
 $\sum_{n=0}^{\infty}\|e_k\| < +\infty$. Then the iterative sequence $\{x^k\}$ and $\{s^k\}$ converge weakly to a solution of the optimization problem (\ref{three-sum}).
\end{thm}

The key implementation of the above iteration scheme (\ref{eq4-1}) is to compute proximity operator of function $\gamma h(Bx)$. Although it still doesn't have a closed-form solution, as we have done in the last section, we can still obtain a precision solution from two approaches: one is dual and the other is primal-dual. In the following, we will present the details. In particular, we will show that the PD3O \cite{yanm2016} is actually a special case of our proposed iterative algorithm.

\subsection{Dual approach}

In this part, we present how to get the updated iterative sequences $s^k$ in (\ref{eq4-1}) via the dual approach. For convenience, let $u^k = 2x^k - z^k - \gamma \nabla f(x^k)$, we have
\begin{equation}\label{eq4-2}
prox_{\gamma (h\circ B)}(u^k ) = \arg\min_{v}\ \{ \frac{1}{2}\| v- u^k \|^2 + \gamma h(Bv) \}.
\end{equation}
Let $g:=0$ and $h:= \gamma h$ in Lemma \ref{key-lemma}, it follows from the derivation of (\ref{eq3-2}) and (\ref{eq3-2-1}), we obtain the dual problem of the minimization  problem (\ref{eq4-2}) is
\begin{equation}\label{eq4-3}
\max_{y\in Y}\ -\frac{1}{2}\| \gamma B^{*}y - u^k \|^2 - \gamma h^{*}(y),
\end{equation}
and the primal optimal solution of (\ref{eq4-2}) $v^{*} = u^k - \gamma B^{*}y^{*}$, where $y^{*}$  is the dual optimal solution of (\ref{eq4-3}). The dual optimization problem (\ref{eq4-3}) can be rewritten as follows,
\begin{equation}\label{eq4-4}
\min_{y\in Y}\ \frac{1}{2}\|  B^{*}y - \frac{1}{\gamma}u^k \|^2 + \frac{1}{\gamma} h^{*}(y),
\end{equation}
Both of the optimization problem (\ref{eq4-3}) and (\ref{eq4-4}) have the same optimal solution $y^{*}$. The corresponding minimization problem (\ref{eq4-4}) can be solved by the forward-backward splitting algorithm and the iteration scheme is presented below. For any $y^{0}\in Y$, choose $0<\lambda < 2/\lambda_{max}(BB^{*})$,
\begin{equation}\label{eq4-5}
y^{j_k +1} = prox_{\frac{\lambda}{\gamma}h^{*}} (y^{j_k} - \lambda B(B^{*}y^{j_k} - \frac{1}{\gamma}u^k)),\ j_k = 0, 1, 2, \cdots.
\end{equation}
 Let $s^k = u^k - \gamma B^{*}y^{J_k}$, where $y^{J_k}$ is the limit point of the iteration scheme (\ref{eq4-5}). For the updated sequence $z^{k+1}$ in (\ref{eq4-1}), we have
\begin{align}
z^{k+1} & = z^k + s^k - x^k, \\
& = z^k + u^k -\gamma B^{*}y^{J_k} - x^k, \\
& = z^k +  2x^k - z^k - \gamma \nabla f(x^k) - \gamma B^{*}y^{J_k} - x^k,\\
& = x^k - \gamma \nabla f(x^k) - \gamma B^{*} y^{J_k}.
\end{align}
In conclusion, we obtain the following iterative algorithm to solve the optimization problem (\ref{three-sum}), which is based on the three operator splitting  scheme (\ref{eq4-1}).

\begin{algorithm}[H]
\caption{A dual three operator splitting  algorithm for solving the optimization problem (\ref{three-sum})}
\begin{algorithmic}\label{dual-three-operator-splitting}
\STATE \textbf{Initialize:}  Given arbitrary $z^{0}\in X$ and $y^{0}\in Y$. Choose $\gamma \in (0,2/L)$ and $\lambda \in (0 , 2/\lambda_{max}(BB^{*}))$.
\STATE 1. (Outer iteration step)\ For $k=0, 1, 2, \cdots$
\STATE \quad $x^{k} = prox_{\gamma g}(z^k)$;
\STATE 2. (Inner iteration step)\ For $j_k = 0, 1, 2, \cdots$
\STATE \quad $y^{j_k +1} = prox_{\frac{\lambda}{\gamma}h^{*}} ( (I-\lambda BB^{*})y^{j_k}  + \frac{\lambda}{\gamma}B (2x^k - z^k - \gamma \nabla f(x^k)) )$;
\STATE  End inner iteration step when stopping criteria reached. Output: $y^{J_k}$.
\STATE 3. Update $s^k = 2x^k - z^k - \gamma \nabla f(x^k) -\gamma B^{*}y^{J_k}$;
\STATE 4. Update $z^{k+1} = z^k + s^k - x^k$;
\STATE 5. End the outer iteration step when some stopping criteria reached.
\end{algorithmic}
\end{algorithm}

Similar to Theorem \ref{convergence-dual-forward-backward}, we can prove the following convergence theorem of Algorithm \ref{dual-three-operator-splitting} in finite-dimensional Hilbert spaces.
\begin{thm}
Let $\gamma \in (0,2/L)$ and $\lambda \in (0 , 2/\lambda_{max}(BB^{*}))$. For any $x^{0}\in X$ and $y^{0}\in Y$, the iterative sequences $\{x^k\}$ and $\{s^k\}$ are generated by Algorithm \ref{dual-three-operator-splitting}. Then the iterative sequences $\{x^k\}$ and $\{s^k\}$  converge to a solution of the optimization problem (\ref{three-sum}).

\end{thm}

\begin{proof}
Let $y$ is an optimal solution of (\ref{eq4-3}), by Lemma \ref{key-lemma}, $v= u^k - \gamma B^{*}y$ is the optimal solution of (\ref{eq4-2}), i.e., $prox_{\gamma (h\circ B)}(u^k) = u^k - \gamma B^{*}y$. Then it follows from the classical convergence of the forward-backward splitting algorithm, the iterative sequence $y^k \rightarrow y$ as $k\rightarrow \infty$. For a given constant $\frac{1}{\gamma \|B\|k^2} >0$, there exists an integer $j_k >0$, when $J_k \geq j_k$, we have $\|y^{J_k}-y\|\leq \frac{1}{\gamma \|B\| k^2}$. Noticing that $s^k = u^k - \gamma B^{*}y^{J_k}$, we have
\begin{align}
\| e_k \| & = \| s^k - prox_{\gamma (h\circ B)}(u^k)  \| \nonumber \\
& = \| u^k - \gamma B^{*}y^{J_k} - (u^k - \gamma B^{*} y )  \| \nonumber \\
& \leq \gamma \|B\| \| y^{J_k} - y \| \leq \frac{1}{k^2}.
\end{align}
Therefore, $\sum_{k=0}^{\infty}\|e_k\| < +\infty$. By Theorem \ref{main-convergence-theorem2}, we can conclude that the iterative sequences $\{x^k\}$ and $\{s^k\}$  converge to a solution of the optimization problem (\ref{three-sum}).

\end{proof}

\subsection{Primal-dual approach}
In this part, we employ the primal-dual proximity algorithm to solve the optimization problem (\ref{eq4-2}). Given arbitrary $v^{0}\in X$ and $y^{0}\in Y$, for $j_k= 0,1, 2, \cdots$, the iteration scheme is defined by
\begin{align}
v^{j_k +1} & = prox_{\tau (\frac{1}{2}\| \cdot - u^k \|^2)} (v^{j_k} - \tau B^{*}y^{j_k}),\label{eq4-6-1} \\
y^{j_k +1} & = prox_{\sigma (\gamma h)^{*}} (y^{j_k} + \sigma B(2v^{j_k +1} - v^{j_k})), \label{eq4-6-2}
\end{align}
where $\tau >0$ and $\sigma >0$ satisfy $\tau \sigma \|B\|^2 <1$. By the definition of proximity operator and Lemma \ref{lem-proximity-conjugate}, the iteration scheme (\ref{eq4-6-1}) and (\ref{eq4-6-2}) can be  simplified as
\begin{align}
v^{j_k +1} & = \frac{v^{j_k} - \tau B^{*}y^{j_k} +\tau u^{k}}{1+\tau},\\
y^{j_k +1} & = \gamma prox_{\frac{\sigma}{\gamma}h^{*}}(\frac{1}{\gamma}y^{j_k} +\frac{\sigma}{\gamma}B(2v^{j_k +1} - v^{j_k}) ).
\end{align}
In conclusion, we obtain the following iterative algorithm to solve the optimization problem (\ref{three-sum}).

\begin{algorithm}[H]
\caption{A primal-dual three operator splitting  algorithm for solving the optimization problem (\ref{three-sum})}
\begin{algorithmic}\label{primal-dual-three-operator-splitting}
\STATE \textbf{Initialize:}  Given arbitrary $z^{0}, v^{0}\in X$ and $y^{0}\in Y$. Choose $\gamma \in (0,2/L)$. Let $\sigma>0$ and $\tau >0$ satisfy the
condition that $\tau \sigma < \frac{1}{\|B\|^2}$.
\STATE 1. (Outer iteration step)\ For $k=0, 1, 2, \cdots$
\STATE \quad 1.a. $x^{k} = prox_{\gamma g}(z^k)$;
\STATE \quad 1.b. $u^{k} = 2x^k - z^k - \gamma \nabla f(x^k)$;
\STATE 2. (Inner iteration step)\ For $j_k = 0, 1, 2, \cdots$
\STATE \quad 2.a. $v^{j_k +1}  = \frac{v^{j_k} - \tau B^{*}y^{j_k} +\tau u^{k}}{1+\tau}$;
\STATE \quad 2.b. $y^{j_k +1}  = \gamma prox_{\frac{\sigma}{\gamma}h^{*}}(\frac{1}{\gamma}y^{j_k} +\frac{\sigma}{\gamma}B(2v^{j_k +1} - v^{j_k}) )$;
\STATE \ End inner iteration when the stopping criteria reached and output $v^{J_k}$
\STATE 3. \quad 1.c. Update $z^{k+1} = z^k + v^{J_k} - x^k$.
\STATE 4. End the outer iteration step when some stopping criteria reached
\end{algorithmic}
\end{algorithm}

We can also prove the convergence of Algorithm \ref{primal-dual-three-operator-splitting} in finite-dimensional Hilbert spaces. The proof method is similar to Theorem \ref{convergence-primal-dual-forward-backward}.

\begin{thm}
Let $\gamma \in (0,2/L)$. Let $\sigma>0$ and $\tau >0$ satisfy the
condition that $\tau \sigma < \frac{1}{\|B\|^2}$. For any $z^{0}, v^{0}\in X$ and $y^{0}\in Y$, the iterative sequences $\{x^k\}$, $\{z^k\}$, $\{v^k\}$ and $\{y^k\}$ are generated by Algorithm \ref{primal-dual-three-operator-splitting}. Then the iterative sequences $\{x^k\}$ and $\{v^k\}$ converge to a solution of the optimization problem (\ref{three-sum}).

\end{thm}

\begin{proof}
In Lemma \ref{key-lemmalemma}, define $f_1 := \frac{1}{2}\|\cdot - u^k \|^2$ and $f_2 := \gamma h$. Let $v$ is the optimal solution of (\ref{eq4-2}), then by Lemma \ref{key-lemmalemma}, for any $\sigma >0$ and $\tau >0$, there exists a vector $y$ such that
\begin{equation}
\begin{aligned}
v & = \frac{v-\tau B^{*}y + \tau u^k}{1+\tau}, \\
y & = prox_{\sigma (\gamma h)^{*}} (y+ \sigma Bx).
\end{aligned}
\end{equation}
It follows from the convergence of the primal-dual proximity algorithm (Theorem 1 of \cite{chambolleandpock2011}), we have $\{v^k\}$ and $\{y^k\}$ converge to $v$ and $y$, respectively. Then for a given constant $\frac{1+\tau}{k^2} >0$, there exits an integer $j_k$, when $J_k \geq j_k$, we have $\|  v^{J_k} - v \| + \tau \|B\| \|y^{J_k}-y\| \leq \frac{1+\tau}{k^2} $. Therefore, we get
\begin{align}
\|e_k\| & = \|s^k - prox_{\gamma (h\circ B)}(u^k) \| \nonumber \\
& = \|  \frac{v^{J_k} - \tau B^{*}y^{J_k} +\tau u^{k}}{1+\tau} - \frac{v-\tau B^{*}y + \tau u^k}{1+\tau} \| \nonumber \\
& \leq \frac{1}{1+\tau} ( \|v^{J_k}-v\| + \tau \|B\| \|y^{J_k}-y\| ) \leq \frac{1}{k^2}.
\end{align}
Then $\sum_{n=0}^{\infty}\|e_k\| < +\infty$. By Theorem \ref{main-convergence-theorem2}, we can conclude that the iterative sequences $\{x^k\}$ and $\{v^k\}$  converge to a solution of the optimization problem (\ref{three-sum}). This completes the proof.

\end{proof}

\subsection{Connection to existing iterative algorithms}
In this part, we show that the proposed Algorithm \ref{dual-three-operator-splitting} recovers the PD3O \cite{yanm2016} proposed by Yan \cite{yanm2016}. In fact, let $j_k =k$ and set the number of inner iteration equals to one in Algorithm \ref{dual-three-operator-splitting}, then we have the iteration scheme of
\begin{equation}\label{alg-PD3O}
 \left\{
\begin{aligned}
x^{k} & = prox_{\gamma g}(z^k),\\
y^{k +1} & = prox_{\frac{\lambda}{\gamma}h^{*}} ( (I-\lambda BB^{*})y^{k}  + \frac{\lambda}{\gamma}B (2x^k - z^k - \gamma \nabla f(x^k)) ), \\
z^{k+1} & = x^k - \gamma \nabla f(x^k) - \gamma B^{*}y^{k +1}.
\end{aligned}
\right.
\end{equation}
Yan \cite{yanm2016} proved the convergence of the PD3O (\ref{alg-PD3O}) under the condition $0<\gamma <2/L$ and $0< \lambda< 1/\lambda_{max}(BB^{*})$. Algorithm \ref{dual-three-operator-splitting} provides a larger range of acceptable parameters $\lambda$ than the PD3O (\ref{alg-PD3O}). Yan \cite{yanm2016} showed that the PD3O (\ref{alg-PD3O}) is equivalent to the three operator splitting algorithm \cite{davis2015} for solving the optimization problem (\ref{three-sum}) when $B=I$. In fact, let $\lambda =1$ and $B= I$ in the PD3O (\ref{alg-PD3O}), then it reduces to
\begin{equation}\label{alg-3O}
 \left\{
\begin{aligned}
x^{k} & = prox_{\gamma g}(z^k),\\
y^{k +1} & = prox_{\frac{1}{\gamma}h^{*}} (  \frac{1}{\gamma} (2x^k - z^k - \gamma \nabla f(x^k)) ), \\
z^{k+1} & = x^k - \gamma \nabla f(x^k) - \gamma y^{k +1},
\end{aligned}
\right.
\end{equation}
which is exactly the three operator splitting algorithm proposed in \cite{davis2015}. In the next, we show that the PDFP (\ref{alg-PDFP}) is  equivalent to the PD3O (\ref{alg-PD3O}). Let $z^k = x^k - \gamma \nabla f(x^k) - \gamma B^{*} y^k$ and $x^k = prox_{\gamma g}(z^k)$ in the iteration scheme (\ref{alg-PDFP}), then we have
\begin{equation}\label{alg-PDFP-3O}
 \left\{
\begin{aligned}
x^k & =  prox_{\gamma g}(z^k), \\
y^{k+1} & =  prox_{\frac{\lambda}{\gamma}h^{*}}(y^k + \frac{\lambda}{\gamma}Bx^k).
\end{aligned}
\right.
\end{equation}
For the iterative sequence $\{y^{k+1}\}$ in (\ref{alg-PDFP-3O}), we obtain
\begin{align}
y^{k+1} & = prox_{\frac{\lambda}{\gamma}h^{*}} ( (I - \lambda B B^{*})y^k + \frac{\lambda}{\gamma}B( x^k + \gamma B^{*}y^k )  ), \nonumber \\
& = prox_{\frac{\lambda}{\gamma}h^{*}} ( (I - \lambda B B^{*})y^k + \frac{\lambda}{\gamma}B( x^k + x^k - \gamma \nabla f(x^k) - z^k )   ), \nonumber \\
& = prox_{\frac{\lambda}{\gamma}h^{*}} ( (I - \lambda B B^{*})y^k + \frac{\lambda}{\gamma}B( 2x^k  - \gamma \nabla f(x^k) - z^k )   ).
\end{align}
Further, let $z^{k+1} = x^k - \gamma \nabla f(x^k) - \gamma B^{*} y^{k+1}$, it follows from $x^k = prox_{\gamma g}(z^k)$ that $x^{k+1} = prox_{\gamma g}(z^{k+1}) = prox_{\gamma g}(x^k - \gamma \nabla f(x^k) - \gamma B^{*} y^{k+1})$, which is the same as the updated iterative sequence $\{x^{k+1}\}$ in the iteration scheme of PDFP (\ref{alg-PDFP}). This confirms that the PDFP \cite{chenpj2016} is equivalent to the PD3O \cite{yanm2016}.  The PD3O \cite{yanm2016} only contains the computation of proximity operator of function $g$ for one time, while the PDFP \cite{chenpj2016} needs to compute it two times.

Let $j_k =k$ and the number of inner iteration equals to one, then the iteration scheme of Algorithm \ref{primal-dual-three-operator-splitting} is reduced to
\begin{equation}\label{alg-ours}
 \left\{
\begin{aligned}
x^{k} & = prox_{\gamma g}(z^k),\\
u^{k} & = 2x^k - z^k - \gamma \nabla f(x^k),\\
v^{k +1} & = \frac{v^{k} - \tau B^{*}y^{k} +\tau u^{k}}{1+\tau},\\
y^{k +1} & = \gamma prox_{\frac{\sigma}{\gamma}h^{*}}(\frac{1}{\gamma}y^{k} +\frac{\sigma}{\gamma}B(2v^{k +1} - v^{k}) ),\\
z^{k+1} & = z^k + v^{k +1} - x^k.
\end{aligned}
\right.
\end{equation}
The iteration scheme (\ref{alg-ours}) is different from PDFP \cite{chenpj2016}, Condat-Vu algorithm \cite{condat2013,vu2013ACM} and PD3O \cite{yanm2016}. To the best of our knowledge, there is no existing iterative algorithm, which is equivalent to the iterative algorithm (\ref{alg-ours}). Besides the before mentioned three iterative algorithms, the iteration scheme (\ref{alg-ours}) could be viewed as the fourth type of iterative algorithm to solve the optimization problem (\ref{three-sum}).

%%%%%%%%%%%%%%%%%%%%%%%%%%%%%%%%%%%%%%%%%%%%%%%%%%%%%%%%%%%%%
\section{Numerical results}
\label{numer_test}

In this section, we study the performance of the proposed iterative algorithms to solve the fused Lasso problem (\ref{fused-lasso}), the constrained total variation regularization problem (\ref{constrained-tv}) and the low-rank total variation image super-resolution problem (\ref{lrtv}). All the experiments are performed in a
standard Lenovo Laptop with Intel(R) Core(TM) i7-4712MQ CPU 2.3GHz and 4GB RAM  under MATLAB (2013a) software.

  \textit{(1) Parameters setting.} Generally speaking, a large selection of iterative parameter $\gamma$ will reduce to fast convergence of the iterative algorithms. So we fix $\gamma = 1.9/L$ for all the proposed iterative algorithms, where $L$ is the Lipschitz constant of $\nabla f$. According to the numerical results of Yan \cite{yanm2016}, the choice of the iterative parameter $\lambda$ has little influence on the convergence speed. Since our proposed iterative algorithms provide a larger range of acceptable parameters of $\lambda$ than the PDFP \cite{chenpj2016} and PD3O \cite{yanm2016}, so we provide two choices of the iterative parameter $\lambda$ in Algorithm \ref{dual-forward-backward} and Algorithm \ref{dual-three-operator-splitting}. We also give two different choice of $\sigma$ and $\tau$ for Algorithm \ref{primal-dual-forward-backward} and Algorithm \ref{primal-dual-three-operator-splitting}. The detailed parameters selection are summarized in Table \ref{para-selection}.

\begin{table}[htbp]
\footnotesize
\centering
\caption{Iterative parameters for the proposed algorithms}
\begin{tabular}{c|c|c}
\hline
Methods & Parameters type I &  Parameters type II \\
\hline
\hline
Algorithm \ref{dual-forward-backward} and Algorithm \ref{dual-three-operator-splitting} & $\lambda_1 = 1.9/\lambda_{max}(BB^{T})$ &  $\lambda_2 = 1/\lambda_{max}(BB^{T})$ \\
\hline
Algorithm \ref{primal-dual-forward-backward} and Algorithm \ref{primal-dual-three-operator-splitting}  & $\sigma_1  = 1/\|B\|^2, \tau_1 = 1$ & $\sigma_2 =\tau_2 = 1/\|B\|$ \\
\hline
\end{tabular}\label{para-selection}
\end{table}

\textit{(2) Performance evaluations.}\ The Signal-to-Noise (SNR) and Normalized Mean Square Distance (NMSD) are used to measure the quality of the reconstructed signal or image, where
%$$
%SNR =10 log \frac{\sum_{i=1}^{n}(x(i)-\overline{x})^2}{\sum_{i=1}^{n}(x(i)-x_r(i))^2},
%$$
%and
%$$
%NMSD = \frac{\sqrt{\sum_{i=1}^{n}(x(i)-x_{r}(i))^2}}{\sqrt{\sum_{i=1}^{n}(x(i)-\overline{x}(i))^2}},
%$$
$$
SNR =20 log \frac{\|x-\overline{x}\|_{2}}{\|x-x_r\|_{2}},
$$
and
$$
NMSD = \frac{\|x-x_r\|_{2}}{\|x-\overline{x}\|_{2}},
$$
where $\overline{x}$ is the average value of the ideal $x$ and $x_r$ is the reconstructed signal or image.

 \textit{(3) Stopping criteria.}\ We set the relative error between two successive iterative sequences is less than a prescribed tolerance value as the stopping criteria. That is
$$
\frac{\|x^{k+1} - x^k \|_2}{\|x^k\|_2} \leq \epsilon,
$$
where $\epsilon$ is a given small number. If the recontruction $x^{k}$ is an image, then the 2-norm is replaced by Frobenius norm, respectively.

\subsection{Fused Lasso problem}

To solve the Fused Lasso problem (\ref{fused-lasso}),  let's define $f(x) = \frac{1}{2}\|Ax-b\|_{2}^{2}$, $g(x) = \mu_1 \|x\|_{1}$ and $h(Bx) = \mu_2 \|Dx\|_{1}$ in the optimization problem (\ref{three-sum}), then we can apply the proposed iterative algorithms. We follow Ye and Xiao \cite{yegb2011}'s method to generate synthetic dataset. The true coefficient $x=(x_1,x_2, \cdots, x_n)\in R^n$ is generated according to
\begin{equation}\label{}
x_i = \left\{
\begin{aligned}
2, & \quad i=1,2,\cdots, 20, 121, 122, \cdots, 125,\\
3, & \quad i = 41, \\
1, & \quad i = 71, 72, \cdots, 85, \\
0, & \quad else.
\end{aligned}
\right.
\end{equation}
$A\in R^{m\times n}$ is a random matrix whose elements follow the standard Gaussian distribution and $b = Ax +e$, where $e$ is additive Gaussian noise with mean $0$ and variance $0.1$. Here, we set $m=100$ and $n=200$. It's known that the Lispchitz constant of $\nabla f$ is $\|A\|^2$ and the eigenvalues of $DD^{T}$ are $2-2cos(i\pi/n)$ \cite{chenpj2016}, $i=1,2, \cdots, n-1$. So we take $\lambda_{max}(DD^{T})=4$ and $\|D\|=2$. The regularization parameters are set as $\mu_1 = 0.2$ and $\mu_2 = 0.8$.

In the first experiment, we set the number of inner iteration equals to one for all the proposed iterative algorithms, including Algorithm \ref{dual-forward-backward}, Algorithm \ref{primal-dual-forward-backward}, Algorithm \ref{dual-three-operator-splitting} and Algorithm \ref{primal-dual-three-operator-splitting}. In this case, Algorithm \ref{dual-forward-backward} and Algorithm \ref{dual-three-operator-splitting} reduce to the corresponding FDFP \cite{chenpj2016} and PD3O \cite{yanm2016}, respectively. Algorithm \ref{primal-dual-forward-backward} reduces to the Condat-Vu algorithm \cite{condat2013,vu2013ACM}. We test the performance of these iterative algorithms with given parameters according to Table \ref{para-selection}.
The numerical results are reported in Table \ref{lasso-results-1}. The symbol $'-'$ in Table \ref{lasso-results-1} means it exceeds the maximum $5000$ iteration numbers. We can see from Table \ref{lasso-results-1} that Algorithm \ref{primal-dual-forward-backward} and Algorithm \ref{primal-dual-three-operator-splitting} perform nearly the same. Algorithm \ref{dual-forward-backward} and Algorithm \ref{dual-three-operator-splitting} take the parameters type I, which do not converge within the required maximum iteration numbers. Under the choice of parameters type II, we find that all these iterative algorithms converge to the same solution when they converge. However, Algorithm \ref{primal-dual-forward-backward} and Algorithm \ref{primal-dual-three-operator-splitting} require more iteration numbers than Algorithm \ref{dual-forward-backward} and Algorithm \ref{dual-three-operator-splitting}. The proposed Algorithm \ref{dual-forward-backward} and Algorithm \ref{dual-three-operator-splitting} perform more robustness in the parameters type II than in the parameters type I. For Algorithm \ref{primal-dual-forward-backward} and Algorithm \ref{primal-dual-three-operator-splitting}, the number of iterations of parameters type I is less than the number of iterations of parameters type II. But this difference is not too much.

\begin{table}[htbp]
\footnotesize
\centering
\caption{Numerical results obtained by  Algorithm \ref{dual-forward-backward}, Algorithm \ref{primal-dual-forward-backward}, Algorithm \ref{dual-three-operator-splitting} and Algorithm \ref{primal-dual-three-operator-splitting} in terms of NMSD, SNR(dB) and the iteration numbers (Iter).}
\begin{tabular}{c|c|ccccccc}
\hline
Parameters  & \multirow{2}[1]{*}{Methods} &  \multicolumn{3}{c}{$ \epsilon = 10^{-4}$} &  & \multicolumn{3}{c}{$ \epsilon = 10^{-8}$} \\ \cline{3-5} \cline{7-9}
type & &  $NMSD$ & $SNR(dB)$ & $Iter$ & & $NMSD$ & $SNR(dB)$ & $Iter$ \\
\hline
\hline
\multirow{4}[1]{*}{$I$}  & Algorithm \ref{dual-forward-backward} &  $0.0063$ & $44.2488$ & $-$ & &  $0.0063$ & $44.2488$ & $-$ \\
 & Algorithm \ref{primal-dual-forward-backward} &  $0.0066$ & $43.6446$ & $757$ & & $0.0060$ & $44.5044$ & $986$ \\
 & Algorithm \ref{dual-three-operator-splitting} &  $0.0067$ & $43.6584$ & $-$ & &  $0.0067$ & $43.6584$ & $-$ \\
 & Algorithm \ref{primal-dual-three-operator-splitting} &  $0.0066$ & $43.6411$ & $758$ & &  $0.0060$ & $44.5044$ & $987$ \\
\hline
\multirow{4}[1]{*}{$II$}  & Algorithm \ref{dual-forward-backward} &  $0.0060$ & $44.4103$ & $387$ & &  $0.0060$ & $44.5044$ & $626$ \\
 & Algorithm \ref{primal-dual-forward-backward} &  $0.0076$ & $42.3389$ & $1119$ & & $0.0060$ & $44.5044$ & $1471$ \\
 & Algorithm \ref{dual-three-operator-splitting} &  $0.0060$ & $44.4048$ & $387$ & &  $0.0060$ & $44.5044$ & $627$ \\
 & Algorithm \ref{primal-dual-three-operator-splitting} &  $0.0076$ & $42.3340$ & $1120$ & &  $0.0060$ & $44.5044$ & $1471$ \\
\hline
\end{tabular}\label{lasso-results-1}
\end{table}

Further, we plot the objective function values and SNR values versus the number of iterations in Figure \ref{fvalue1} and Figure \ref{snr1}, respectively. Figure \ref{signal1} shows the recovered signal and the true sparse signal using Algorithm \ref{dual-forward-backward}, Algorithm \ref{primal-dual-forward-backward}, Algorithm \ref{dual-three-operator-splitting} and Algorithm \ref{primal-dual-three-operator-splitting} which taking parameters type I.

%-------------
   \begin{figure}
      \setlength{\abovecaptionskip}{-10pt}
   \begin{center}
   \begin{tabular}{c}
   \scalebox{0.65}{\includegraphics{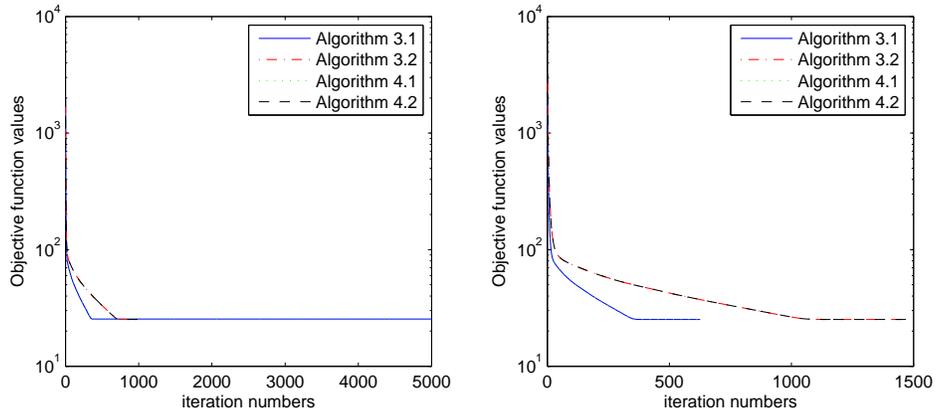}}
   \end{tabular}
   \end{center}
   \caption[]{The objective function values versus the number of iterations for Algorithm \ref{dual-forward-backward}, Algorithm \ref{primal-dual-forward-backward}, Algorithm \ref{dual-three-operator-splitting} and Algorithm \ref{primal-dual-three-operator-splitting}. The left figure is obtained from parameters type I and the right figure is obtained from parameters type II.}
%>>>> use \label inside caption to get Fig. number with \ref{}
   { \label{fvalue1}}
   \end{figure}

%-------------
   \begin{figure}
      \setlength{\abovecaptionskip}{-10pt}
   \begin{center}
   \begin{tabular}{c}
   \scalebox{0.65}{\includegraphics{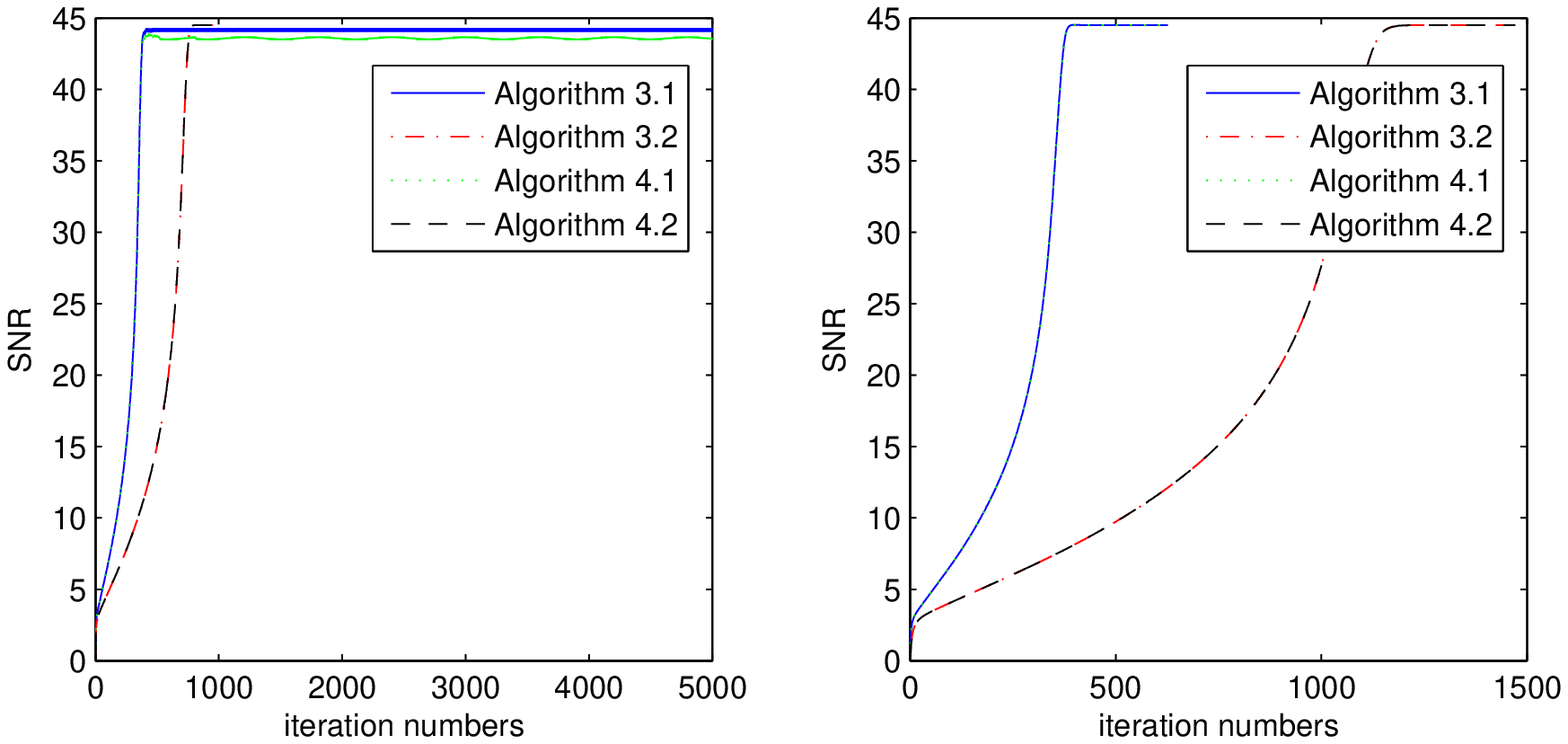}}
   \end{tabular}
   \end{center}
   \caption[]{The SNR values versus the number of iterations for Algorithm \ref{dual-forward-backward}, Algorithm \ref{primal-dual-forward-backward}, Algorithm \ref{dual-three-operator-splitting} and Algorithm \ref{primal-dual-three-operator-splitting}. The left figure is obtained from parameters type I and the right figure is obtained from parameters type II.}
%>>>> use \label inside caption to get Fig. number with \ref{}
   { \label{snr1}}
   \end{figure}

%-------------
   \begin{figure}
      \setlength{\abovecaptionskip}{-10pt}
   \begin{center}
   \begin{tabular}{c}
   \scalebox{0.65}{\includegraphics{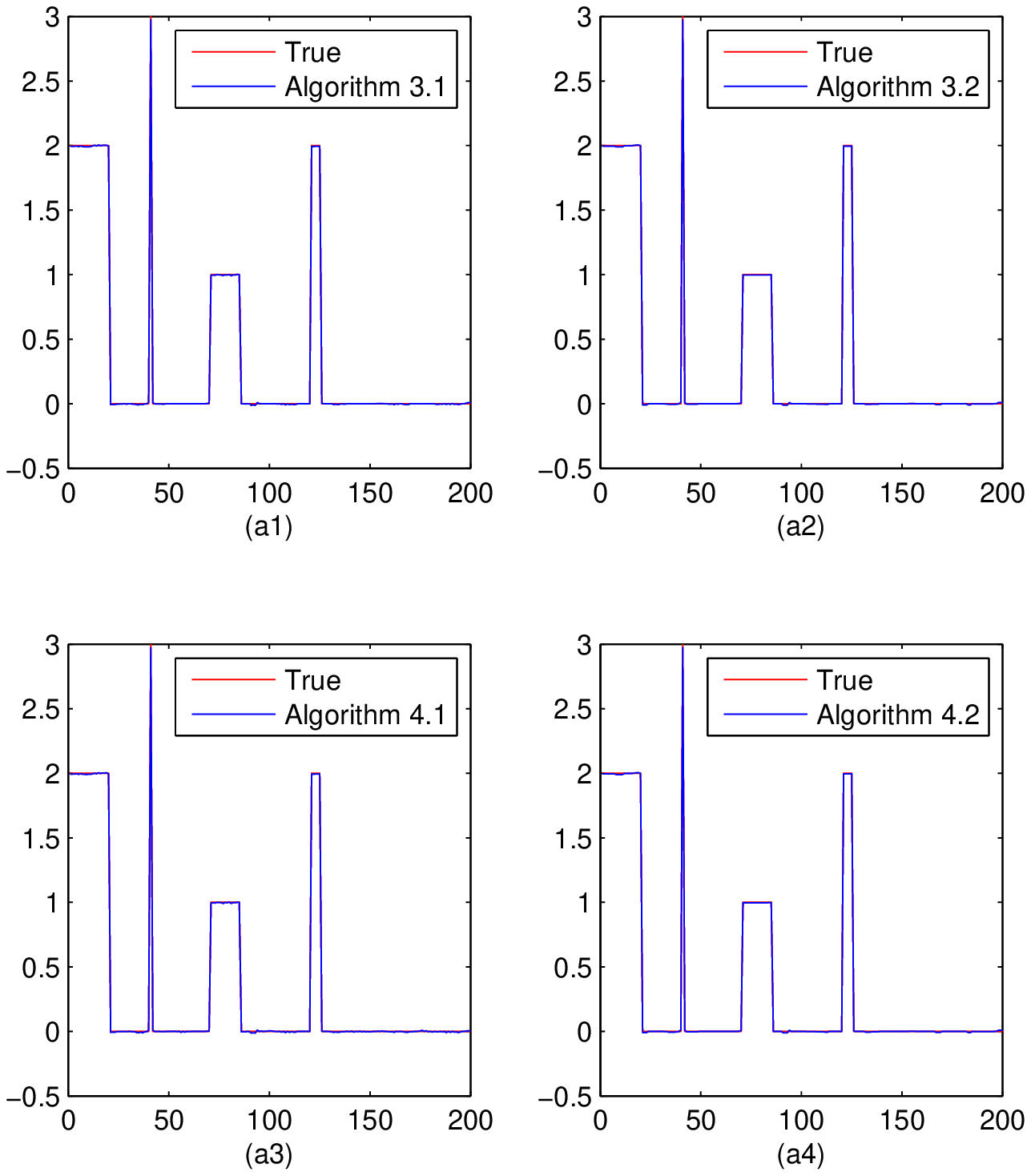}}
   \end{tabular}
   \end{center}
   \caption[]{The true sparse signal and the reconstructed results by using Algorithm \ref{dual-forward-backward}, Algorithm \ref{primal-dual-forward-backward}, Algorithm \ref{dual-three-operator-splitting} and Algorithm \ref{primal-dual-three-operator-splitting}.}
%>>>> use \label inside caption to get Fig. number with \ref{}
   { \label{signal1}}
   \end{figure}

In the second experiment, we demonstrate how the performance of the proposed iterative algorithms is influenced by  the number of inner iterations.
The results are reported in Table \ref{lasso-results-inner-1} and Table \ref{lasso-results-inner-2}, respectively. We can see from Table \ref{lasso-results-1} that the Algorithm \ref{dual-forward-backward} and Algorithm \ref{dual-three-operator-splitting} do not converge within the given maximum iteration numbers when the  parameters type I is selected. In Table \ref{lasso-results-inner-1}, when the number of inner iterations exceeds one, we find that the Algorithm \ref{dual-forward-backward} and Algorithm \ref{dual-three-operator-splitting} converge. It can be seen from Table \ref{lasso-results-inner-1} and Table \ref{lasso-results-inner-2} that by increasing the number of inner iterations, the number of outer iterations required by all the proposed iterative algorithms is reduced, but this decreasing trend will stop over a certain value.

\begin{table}[htbp]
\footnotesize
\centering
\caption{Numerical results obtained by  Algorithm \ref{dual-forward-backward}, Algorithm \ref{primal-dual-forward-backward}, Algorithm \ref{dual-three-operator-splitting} and Algorithm \ref{primal-dual-three-operator-splitting} with the choice of parameters type I.}
\begin{tabular}{c|c|ccccccc}
\hline
Inner iteration  & \multirow{2}[1]{*}{Methods} &  \multicolumn{3}{c}{$ \epsilon = 10^{-4}$} &  & \multicolumn{3}{c}{$ \epsilon = 10^{-8}$} \\ \cline{3-5} \cline{7-9}
numbers & &  $NMSD$ & $SNR(dB)$ & $Iter$ & & $NMSD$ & $SNR(dB)$ & $Iter$ \\
\hline
\hline
\multirow{4}[1]{*}{$2$}  & Algorithm \ref{dual-forward-backward} &  $0.0062$ & $44.2197$ & $385$ & &  $0.0060$ & $44.5044$ & $500$ \\
 & Algorithm \ref{primal-dual-forward-backward} &  $0.0063$ & $44.0661$ & $510$ & & $0.0060$ & $44.5044$ & $659$ \\
 & Algorithm \ref{dual-three-operator-splitting} &  $0.0061$ & $44.2254$ & $385$ & &  $0.0060$ & $44.5044$ & $500$ \\
 & Algorithm \ref{primal-dual-three-operator-splitting} &  $0.0063$ & $44.0643$ & $510$ & &  $0.0060$ & $44.5044$ & $659$ \\
\hline
\multirow{4}[1]{*}{$10$}  & Algorithm \ref{dual-forward-backward} &  $0.0062$ & $44.1763$ & $385$ & &  $0.0060$ & $44.5044$ & $505$ \\
 & Algorithm \ref{primal-dual-forward-backward} &  $0.0062$ & $44.2187$ & $386$ & & $0.0060$ & $44.5044$ & $505$ \\
 & Algorithm \ref{dual-three-operator-splitting} &  $0.0062$ & $44.1763$ & $385$ & &  $0.0060$ & $44.5044$ & $505$ \\
 & Algorithm \ref{primal-dual-three-operator-splitting} &  $0.0062$ & $44.2176$ & $386$ & &  $0.0060$ & $44.5044$ & $505$ \\
\hline
\multirow{4}[1]{*}{$20$}  & Algorithm \ref{dual-forward-backward} &  $0.0062$ & $44.1723$ & $385$ & &  $0.0060$ & $44.5044$ & $506$ \\
 & Algorithm \ref{primal-dual-forward-backward} &  $0.0062$ & $44.1762$ & $385$ & & $0.0060$ & $44.5044$ & $505$ \\
 & Algorithm \ref{dual-three-operator-splitting} &  $0.0062$ & $44.1720$ & $385$ & &  $0.0060$ & $44.5044$ & $506$ \\
 & Algorithm \ref{primal-dual-three-operator-splitting} &  $0.0062$ & $44.1752$ & $385$ & &  $0.0060$ & $44.5044$ & $505$ \\
 \hline
\end{tabular}\label{lasso-results-inner-1}
\end{table}

\begin{table}[htbp]
\footnotesize
\centering
\caption{Numerical results obtained by  Algorithm \ref{dual-forward-backward}, Algorithm \ref{primal-dual-forward-backward}, Algorithm \ref{dual-three-operator-splitting} and Algorithm \ref{primal-dual-three-operator-splitting} with the choice of parameters type II.}
\begin{tabular}{c|c|ccccccc}
\hline
Inner iteration  & \multirow{2}[1]{*}{Methods} &  \multicolumn{3}{c}{$ \epsilon = 10^{-4}$} &  & \multicolumn{3}{c}{$ \epsilon = 10^{-8}$} \\ \cline{3-5} \cline{7-9}
numbers & &  $NMSD$ & $SNR(dB)$ & $Iter$ & & $NMSD$ & $SNR(dB)$ & $Iter$ \\
\hline
\hline
\multirow{4}[1]{*}{$2$}  & Algorithm \ref{dual-forward-backward} &  $0.0061$ & $44.3020$ & $386$ & &  $0.0060$ & $44.5044$ & $506$ \\
 & Algorithm \ref{primal-dual-forward-backward} &  $0.0066$ & $43.6520$ & $683$ & & $0.0060$ & $44.5044$ & $895$ \\
 & Algorithm \ref{dual-three-operator-splitting} &  $0.0061$ & $44.2980$ & $386$ & &  $0.0060$ & $44.5044$ & $500$ \\
 & Algorithm \ref{primal-dual-three-operator-splitting} &  $0.0066$ & $43.6464$ & $683$ & &  $0.0060$ & $44.5044$ & $895$ \\
\hline
\multirow{4}[1]{*}{$10$}  & Algorithm \ref{dual-forward-backward} &  $0.0062$ & $44.1838$ & $385$ & &  $0.0060$ & $44.5044$ & $505$ \\
 & Algorithm \ref{primal-dual-forward-backward} &  $0.0062$ & $44.1873$ & $392$ & & $0.0060$ & $44.5044$ & $514$ \\
 & Algorithm \ref{dual-three-operator-splitting} &  $0.0062$ & $44.1826$ & $385$ & &  $0.0060$ & $44.5044$ & $505$ \\
 & Algorithm \ref{primal-dual-three-operator-splitting} &  $0.0062$ & $44.1864$ & $392$ & &  $0.0060$ & $44.5044$ & $514$ \\
\hline
\multirow{4}[1]{*}{$20$}  & Algorithm \ref{dual-forward-backward} &  $0.0062$ & $44.1756$ & $385$ & &  $0.0060$ & $44.5044$ & $506$ \\
 & Algorithm \ref{primal-dual-forward-backward} &  $0.0062$ & $44.2183$ & $386$ & & $0.0060$ & $44.5044$ & $506$ \\
 & Algorithm \ref{dual-three-operator-splitting} &  $0.0062$ & $44.1747$ & $385$ & &  $0.0060$ & $44.5044$ & $506$ \\
 & Algorithm \ref{primal-dual-three-operator-splitting} &  $0.0062$ & $44.2179$ & $386$ & &  $0.0060$ & $44.5044$ & $506$ \\
 \hline
\end{tabular}\label{lasso-results-inner-2}
\end{table}

Take inner iteration numbers ten and parameters type I, Figure \ref{fvalue_snr1} shows the objective function values and the SNR values versus the number of iterations. We can see from Figure \ref{fvalue_snr1} that the convergence of the four iterative algorithms is nearly the same.

%-------------
   \begin{figure}
      \setlength{\abovecaptionskip}{-10pt}
   \begin{center}
   \begin{tabular}{c}
   \scalebox{0.65}{\includegraphics{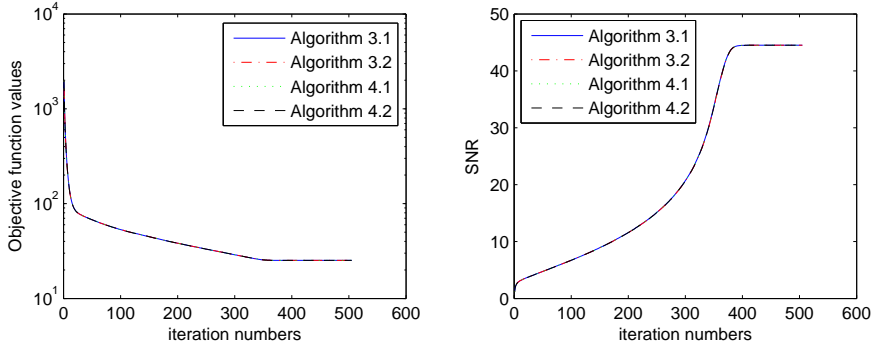}}
   \end{tabular}
   \end{center}
   \caption[]{The comparison results of the four iterative algorithms. }
%>>>> use \label inside caption to get Fig. number with \ref{}
   { \label{fvalue_snr1}}
   \end{figure}

\subsection{Constrained total variation regularization problem}

In this subsection, we employ the proposed iterative algorithms to solve the constrained total variation regularization problem (\ref{constrained-tv}) arising in computed tomography (CT) image reconstruction.
In the following, let's recall the discrete definition of the total variation.

\begin{definition}(\cite{barberoandsr12010,Micchelliandshen2011})\label{tv-definition}
Let $x\in R^{n}$ be a $\sqrt{n}\times \sqrt{n}$ image. The norm $\|\cdot\|_{2}$ and $\|\cdot\|_{1}$ denote the usual 2-norm and 1-norm of vectors, respectively. First, we define a first order difference matrix $B$ as follows,
$$
B_{n \times n} = \left(
                       \begin{array}{ccccc}
                         -1 & 1 & 0 & \cdots & 0 \\
                         0 & -1 & 1 & \cdots & 0 \\
                          &  & \cdots &  &  \\
                         0 & 0 & \cdots & 0 & 0 \\
                       \end{array}
                     \right),
$$
and the matrix $D$ is defined by
\begin{equation}\label{difference}
D = \left(
      \begin{array}{c}
        I \otimes B \\
        B \otimes I \\
      \end{array}
    \right),
\end{equation}
where $I$ denotes the identity matrix and $\otimes$ denotes the Kronecker inner product.

(i)\  The isotropic total variation (ITV) is defined by $\|x\|_{ITV} = \|Dx\|_{2,1}$, where
$\|y\|_{2,1} = \sum_{i=1}^{n}\sqrt{y_{i}^{2}+y_{n+i}^{2}}$, $y\in R^{2n\times 1}$;

(ii)\ The anisotropic total variation (ATV) is defined by $\|x\|_{ATV} = \|Dx\|_{1}$.

\end{definition}

In this case, let $f(x)=\frac{1}{2}\|Ax-b\|_{2}^{2}$, $g(x)=\delta_{C}(x)$ ($\delta_{C}(x)$ is the indicator function of the closed convex set $C$), and $h(Bx)=\mu \|x\|_{TV}$. The closed convex set $C$ is set as nonnegative set, i.e., $C = \{ x\in R^n | x_i \geq 0 \}$.
We use the standard Shepp-Logan phantom as the reconstructed image (See Figure \ref{shepp-logan}). The phantom is scanned by fan-beam way with 20 views distributed randomly from 0 to 360 and 320 rays in each view. So the size of the system matrix $A$ is $6400\times 65536$. The simulated projection data is generated by AIRtools \cite{hansen1}. The data vector $b$ is assumed to be corrupted by random Gaussian noise with zero mean and 0.01 variance. For the iterative parameters, we let $\lambda = 1/\lambda_{max}(DD^{T})$ for Algorithm \ref{dual-forward-backward} and Algorithm \ref{dual-three-operator-splitting}, and $\sigma = 1/\lambda_{max}(DD^{T}), \tau =1$ for Algorithm \ref{primal-dual-forward-backward} and Algorithm \ref{primal-dual-three-operator-splitting}. It is known that the discrete gradient matrix $D$ (\ref{difference}) has $\lambda_{max}(DD^{T})=8$. The Lipschitz constant of $\nabla f(x)$ is estimated via the power iteration method. The regularization parameter $\mu = 0.5$ is used.
The numerical results are reported in Table \ref{tv-results}.

We can see from Table \ref{tv-results} that the performance of Algorithm \ref{dual-forward-backward} is the same as Algorithm \ref{dual-three-operator-splitting}. This result confirms that the Algorithm \ref{dual-forward-backward} and Algorithm \ref{dual-three-operator-splitting} are equivalent. Similarly, Algorithm \ref{primal-dual-forward-backward} and Algorithm \ref{primal-dual-three-operator-splitting} are the same. When the number of inner iterations is one, the primal-dual based Algorithm \ref{primal-dual-forward-backward} and Algorithm \ref{primal-dual-three-operator-splitting} require more iterations than the Algorithm \ref{dual-forward-backward} and Algorithm \ref{dual-three-operator-splitting}, which are proposed by the dual method. By increasing the number of inner iterations, the number of outer iterations required to reach the same accuracy by Algorithm \ref{dual-forward-backward}, Algorithm \ref{primal-dual-forward-backward}, Algorithm \ref{dual-three-operator-splitting} and Algorithm \ref{primal-dual-three-operator-splitting}  is getting closer. When the number of inner iterations is ten, we can see that the four iterative algorithms converge to the optimal solution at the same rate. The corresponding objective function values are shown in the Figure \ref{tv-fvalue-snr} left, and the SNR values versus the number of iterations is plotted in the Figure \ref{tv-fvalue-snr} right. The reconstructed images of the four iterative algorithms are shown in Figure \ref{tv-rec-image}.

%-------------
   \begin{figure}
      \setlength{\abovecaptionskip}{-20pt}
   \begin{center}
   \begin{tabular}{c}
   \scalebox{0.65}{\includegraphics{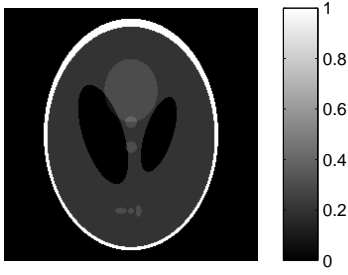}}
   \end{tabular}
   \end{center}
   \caption[]{The standard $256\times 256$ Shepp-Logan phantom.}
%>>>> use \label inside caption to get Fig. number with \ref{}
   { \label{shepp-logan}}
   \end{figure}

\begin{table}[htbp]
\scriptsize
\centering
\caption{Numerical results obtained by  Algorithm \ref{dual-forward-backward}, Algorithm \ref{primal-dual-forward-backward}, Algorithm \ref{dual-three-operator-splitting} and Algorithm \ref{primal-dual-three-operator-splitting} for solving the constrained TV problem in CT image reconstruction.}
\begin{tabular}{c|c|ccccccccccc}
\hline
Inner iteration  & \multirow{2}[1]{*}{Methods} &  \multicolumn{3}{c}{$ \epsilon = 10^{-4}$} &  & \multicolumn{3}{c}{$ \epsilon = 10^{-6}$} &  & \multicolumn{3}{c}{$ \epsilon = 10^{-8}$}\\ \cline{3-5} \cline{7-9} \cline{11-13}
numbers & &  $NMSD$ & $SNR(dB)$ & $Iter$ & & $NMSD$ & $SNR(dB)$ & $Iter$ & & $NMSD$ & $SNR(dB)$ & $Iter$ \\
\hline
\hline
\multirow{4}[1]{*}{$1$}  & Algorithm \ref{dual-forward-backward} &  $0.1756$ & $15.1103$ & $1036$ & &  $0.0213$ & $33.4403$ & $9524$ & &  $0.0209$ & $33.6169$ & $14615$\\
 & Algorithm \ref{primal-dual-forward-backward} &  $0.2344$ & $12.6014$ & $1021$ & & $0.0219$ & $33.2001$ & $16799$ & &  $0.0209$ & $33.6160$ & $27983$ \\
 & Algorithm \ref{dual-three-operator-splitting} &  $0.1756$ & $15.1102$ & $1036$ & &  $0.0213$ & $33.4404$ & $9523$ & &  $0.0209$ & $33.6169$ & $14615$ \\
 & Algorithm \ref{primal-dual-three-operator-splitting} &  $0.2344$ & $12.6014$ & $1021$ & &  $0.0219$ & $33.2002$ & $16799$ & &  $0.0209$ & $33.6160$ & $27982$ \\
\hline
\multirow{4}[1]{*}{$2$}  & Algorithm \ref{dual-forward-backward} &  $0.1756$ & $15.1071$ & $1035$ & &  $0.0213$ & $33.4399$ & $9535$ & &  $0.0209$ & $33.6169$ & $14635$\\
 & Algorithm \ref{primal-dual-forward-backward} &  $0.2035$ & $13.8278$ & $1004$ & & $0.0217$ & $33.2875$ & $11600$ & &  $0.0209$ & $33.6166$ & $19156$ \\
 & Algorithm \ref{dual-three-operator-splitting} &  $0.1756$ & $15.1071$ & $1035$ & &  $0.0213$ & $33.4399$ & $9534$ & &  $0.0209$ & $33.6169$ & $14634$ \\
 & Algorithm \ref{primal-dual-three-operator-splitting} &  $0.2035$ & $13.8278$ & $1004$ & &  $0.0217$ & $33.2874$ & $11599$ & &  $0.0209$ & $33.6166$ & $19156$ \\
\hline
\multirow{4}[1]{*}{$10$}  & Algorithm \ref{dual-forward-backward} &  $0.1757$ & $15.1038$ & $1034$ & &  $0.0213$ & $33.4396$ & $9545$ & &  $0.0209$ & $33.6169$ & $14652$\\
 & Algorithm \ref{primal-dual-forward-backward} &  $0.1758$ & $15.0997$ & $1034$ & & $0.0213$ & $33.4388$ & $9548$ & &  $0.0209$ & $33.6169$ & $14665$ \\
 & Algorithm \ref{dual-three-operator-splitting} &  $0.1757$ & $15.1038$ & $1034$ & &  $0.0213$ & $33.4397$ & $9545$ & &  $0.0209$ & $33.6169$ & $14652$ \\
 & Algorithm \ref{primal-dual-three-operator-splitting} &  $0.1758$ & $15.0997$ & $1034$ & &  $0.0213$ & $33.4388$ & $9548$ & &  $0.0209$ & $33.6169$ & $14665$ \\
\hline
\end{tabular}\label{tv-results}
\end{table}

%-------------
   \begin{figure}
      \setlength{\abovecaptionskip}{-10pt}
   \begin{center}
   \begin{tabular}{c}
   \scalebox{0.65}{\includegraphics{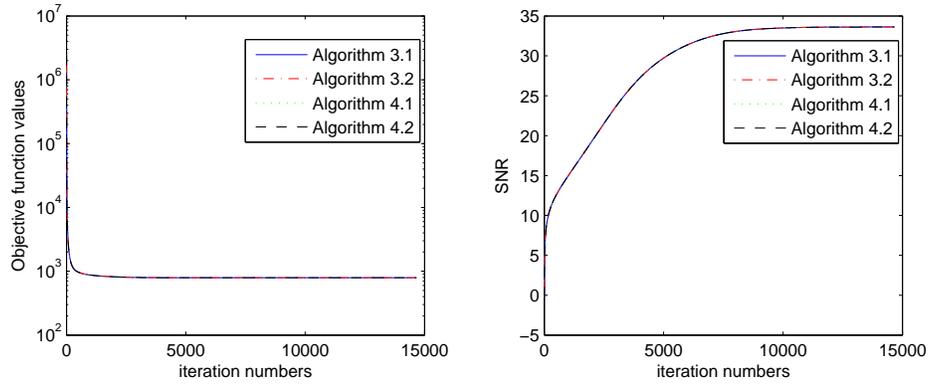}}
   \end{tabular}
   \end{center}
   \caption[]{The comparison results of the four iterative algorithms in terms of the objective function values and SNR values versus the number of iterations in CT image reconstruction.}
%>>>> use \label inside caption to get Fig. number with \ref{}
   { \label{tv-fvalue-snr}}
   \end{figure}

%-------------
   \begin{figure}
      \setlength{\abovecaptionskip}{-20pt}
   \begin{center}
   \begin{tabular}{c}
   \scalebox{0.8}{\includegraphics{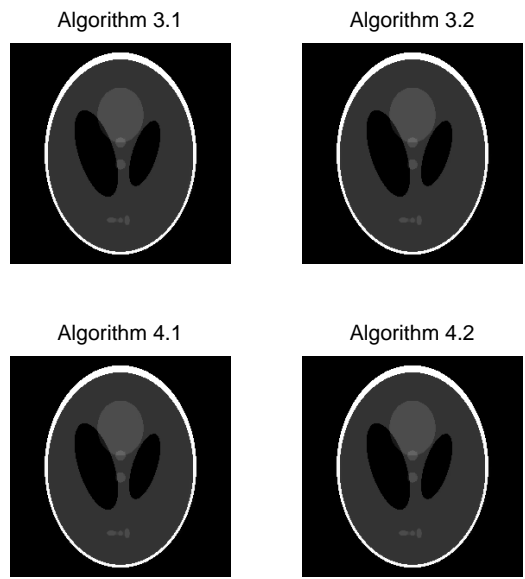}}
   \end{tabular}
   \end{center}
   \caption[]{The reconstructed images by Algorithm \ref{dual-forward-backward}, Algorithm \ref{primal-dual-forward-backward}, Algorithm \ref{dual-three-operator-splitting} and Algorithm \ref{primal-dual-three-operator-splitting} in CT image reconstruction.}
%>>>> use \label inside caption to get Fig. number with \ref{}
   { \label{tv-rec-image}}
   \end{figure}

\subsection{Low-rank total variation image super-resolution problem}

High-resolution (HR) images are always needed in various medical imaging diagnoses, such as CT, MRI, and PET et al. However, in practice, due to the limitations of the imaging acquisition systems or to fast the medical imaging reconstruction time, lower-resolution (LR) images are obtained sometimes.
There is much demand to generate high-resolution images from lower-resolution images. Super-resolution (SR) image reconstruction is the technique that can achieve this goal. In this subsection, we apply the proposed iterative algorithms to solve the low rank total variation (LRTV) image super-resolution optimization problem (\ref{lrtv}). In the LRTV optimization problem (\ref{lrtv}), there is two regularization terms, one is the nuclear norm $\|X\|_{*}$, which is a convex relaxation of lower rank constraint. The other is the total variation $\|X\|_{TV}$, which can be represented by a combination of a convex function with a linear operator. It is easy to notice that the  LRTV optimization problem (\ref{lrtv}) is a special case of the optimization problem (\ref{three-sum}) by letting $f(X) = \frac{1}{2}\| DSX - T \|_{F}^{2}$, $g(X)=\lambda_1 \|X\|_{*}$ and $h(BX) = \lambda_2 \|X\|_{TV}$. It is worth mentioning that the proximity operator of nuclear norm has a closed-form solution according to \cite{Caijf2010SIAM}.

In the numerical experiment, we use the same data-set as \cite{shif2015}. In detail, a representative 2D slice from T1 MR phantom in Brainweb\footnote{http://www.bic.mni.mcgill.ca/brainweb} is selected, which has a size of $217\times 181$ with a resolution of 1mm (Figure \ref{Ori_HR_image}).
%-------------
   \begin{figure}
      \setlength{\abovecaptionskip}{-20pt}
   \begin{center}
   \begin{tabular}{c}
   \scalebox{0.8}{\includegraphics{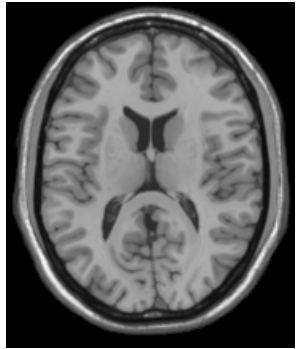}}
   \end{tabular}
   \end{center}
   \caption[]{Original high-resolution image.}
%>>>> use \label inside caption to get Fig. number with \ref{}
   { \label{Ori_HR_image}}
   \end{figure}
The blurring operator $S$ is implemented using a Gaussian Kernel with a standard deviation of 1 pixel. The blurred image is then down-sampled by averaging every 4 pixels. The upsampled operator is implemented by nearest-neighbor interpolation. Besides SNR and NMSD, we also use Structural Similarity Index (SSIM) \cite{wangzhou} to evaluate the quality of reconstruction images. The SSIM is defined as,
$$
SSIM(f,g) = \frac{(2\mu_{f}\mu_{g}+c_1)(2\sigma_{fg}+c_2)}{(\mu_{f}^{2}+\mu_{g}^{2}+c_1)(\sigma_{f}^{2}+\sigma_{g}^{2}+c_1)},
$$
where $\mu_f$ and $\mu_g$ are the mean values respectively in the original HR image $f$ and recovered image $g$, $\sigma_{f}^{2}$ and $\sigma_{g}^{2}$ are the variances, $\sigma_{fg}$ is the covariance of two images, $c_1 = (k_1 L)^2$ and $c_2 = (k_2 L)^2$ with $k_1 = 0.01$ and $k_2 = 0.03$, and $L$ is the dynamic range of pixel values. SSIM ranges from 0 to 1, and 1 means perfect recovery.

We test the performance of the proposed iterative algorithms with a choice of regularization parameters $\lambda_1 = 0.01$ and $\lambda_2 =0.01$ according to \cite{shif2015}. In this numerical tests, we don't know exactly the operator norm of the down-sampling operator $D$ and the blurring operator $S$, so we tune the parameter $\gamma$ as far as possible to ensure that the proposed iterative algorithms converge. We set $\gamma = 0.1$. The other iterative parameters are chosen the same as Subsection 5.2. The numerical results are summarized in Table \ref{lrtv-results}. They symbol $'-'$ means that the maximum iteration number $1\times 10^{5}$ exceeds. We can observe that the proposed iterative algorithms perform the same in terms of NMSD, SNR and SSIM when the stopping criterion $\epsilon = 10^{-8}$ and the number of inner iteration number equals to ten. The corresponding SNR and objective function values versus iteration numbers for all the proposed algorithms are plotted in Figure \ref{SR_10_fvalue_snr}. We can see from Figure \ref{SR_10_fvalue_snr} that the SNR is not monotonically increasing, but first increases and then decreases. For all the proposed algorithms, the maximum SNR is the same and equals to $22.2399(dB)$. The number of iterations used is $48671$, $48718$, $48671$ and $48718$, respectively. This result is consistent with our previous numerical results.
In order to have visual inspection, Figure \ref{SR_10_image} shows the reconstructed images for the proposed iterative algorithms.

\begin{table}[htbp]
\scriptsize
\centering
\caption{Numerical results of the proposed iterative algorithms for solving the low-rank total variation image super-resolution optimization problem (\ref{lrtv}).}
\begin{tabular}{c|c|ccccccccc}
\hline
Inner iteration  & \multirow{2}[1]{*}{Methods} &  \multicolumn{4}{c}{$ \epsilon = 10^{-6}$} &  & \multicolumn{4}{c}{$ \epsilon = 10^{-8}$}  \\ \cline{3-6} \cline{8-11}
numbers & &  $NMSD$ & $SNR(dB)$ & $SSIM$ & $Iter$ & & $NMSD$ & $SNR(dB)$ & $SSIM$ & $Iter$  \\
\hline
\hline
\multirow{4}[1]{*}{$1$}  & Algorithm \ref{dual-forward-backward} &  $0.0781$ & $22.1440$ & $0.9796$ & $34573$  & &  $0.0797$ & $21.9701$ & $0.9642$ &  $-$ \\
 & Algorithm \ref{primal-dual-forward-backward} &    $0.0890$ & $21.0146$ & $0.9762$ & $19046$  & &  $0.0773$ & $22.2393$ & $0.9635$ &  $-$ \\
 & Algorithm \ref{dual-three-operator-splitting} &  $0.0781$ & $22.1440$ & $0.9796$ & $34573$  & &  $0.0797$ & $21.9701$ & $0.9642$ &  $-$ \\
 & Algorithm \ref{primal-dual-three-operator-splitting} &  $0.0890$ & $21.0146$ & $0.9762$ & $19046$  & &  $0.0773$ & $22.2393$ & $0.9635$ &  $-$ \\
\hline
\multirow{4}[1]{*}{$2$}  & Algorithm \ref{dual-forward-backward} &  $0.0781$ & $22.1440$ & $0.9602$ & $34573$  & &  $0.0797$ & $21.9701$ & $0.9642$ &  $-$ \\
 & Algorithm \ref{primal-dual-forward-backward} &    $0.0814$ & $21.7906$ & $0.9549$ & $29252$  & &  $0.0783$ & $22.1244$ & $0.9650$ &  $-$ \\
 & Algorithm \ref{dual-three-operator-splitting} &  $0.0781$ & $22.1440$ & $0.9602$ & $34573$  & &  $0.0797$ & $21.9701$ & $0.9642$ &  $-$ \\
 & Algorithm \ref{primal-dual-three-operator-splitting} &  $0.0814$ & $21.7906$ & $0.9549$ & $29252$  & &  $0.0783$ & $22.1244$ & $0.9650$ &  $-$ \\
\hline
\multirow{4}[1]{*}{$10$}  & Algorithm \ref{dual-forward-backward} &  $0.0781$ & $22.1440$ & $0.9602$ & $34573$  & &  $0.0797$ & $21.9701$ & $0.9642$ &  $-$ \\
 & Algorithm \ref{primal-dual-forward-backward} &    $0.0781$ & $22.1435$ & $0.9602$ & $34576$  & &  $0.0797$ & $21.9707$ & $0.9642$ &  $-$ \\
 & Algorithm \ref{dual-three-operator-splitting} &  $0.0781$ & $22.1440$ & $0.9602$ & $34573$  & &  $0.0797$ & $21.9701$ & $0.9642$ &  $-$ \\
 & Algorithm \ref{primal-dual-three-operator-splitting} &  $0.0781$ & $22.1435$ & $0.9602$ & $34576$  & &  $0.0797$ & $21.9707$ & $0.9642$ &  $-$ \\
\hline
\end{tabular}\label{lrtv-results}
\end{table}

%-------------
   \begin{figure}
      \setlength{\abovecaptionskip}{-10pt}
   \begin{center}
   \begin{tabular}{c}
   \scalebox{0.7}{\includegraphics{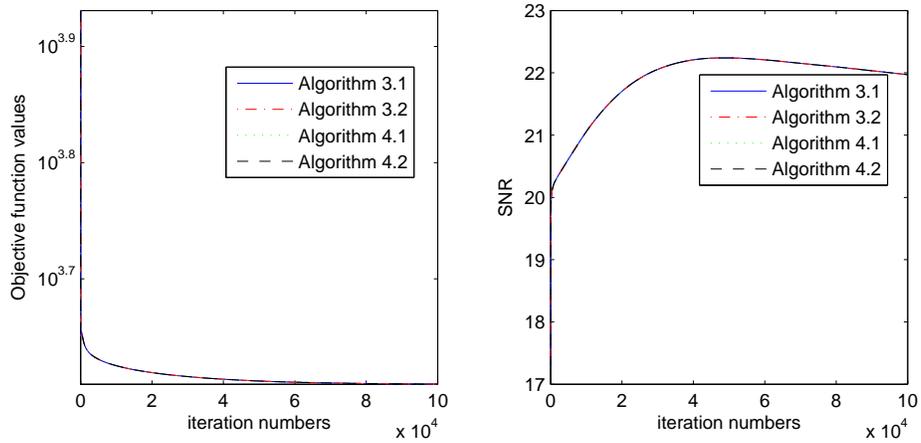}}
   \end{tabular}
   \end{center}
   \caption[]{Objective function values versus iterations (left) and SNR versus iterations (right) in low rank total variation image super-resolution. }
%>>>> use \label inside caption to get Fig. number with \ref{}
   { \label{SR_10_fvalue_snr}}
   \end{figure}

%-------------
   \begin{figure}
      \setlength{\abovecaptionskip}{-10pt}
   \begin{center}
   \begin{tabular}{c}
   \scalebox{0.8}{\includegraphics{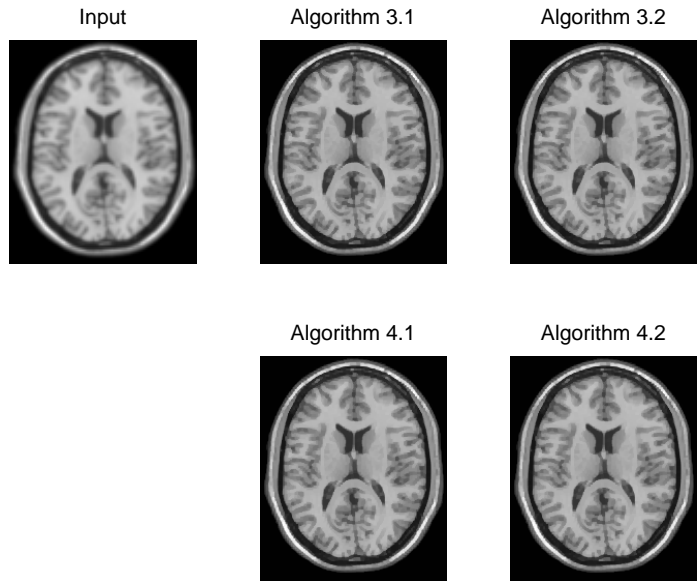}}
   \end{tabular}
   \end{center}
   \caption[]{The reconstructed images by Algorithm \ref{dual-forward-backward}, Algorithm \ref{primal-dual-forward-backward}, Algorithm \ref{dual-three-operator-splitting} and Algorithm \ref{primal-dual-three-operator-splitting} in low rank total variation image super-resolution. }
%>>>> use \label inside caption to get Fig. number with \ref{}
   { \label{SR_10_image}}
   \end{figure}

%%%%%%%%%%%%%%%%%%%%%%%%%%%%%%%%%%%%%%%%%%%%%%%%%%%%%%%%%%%%%
\section{Conclusions}

In this paper, we study a class of convex optimization problem (\ref{three-sum}), which minimizes the sum of three convex functions. We make full use of the gradient of the differentiable convex function of the objective function. To derive effective iterative algorithms for solving the considered optimization problem (\ref{three-sum}), we employ two monotone operator splitting methods: the forward-backward splitting method and the three operator splitting method. In both of the iteration schemes, we are required to compute proximity operator of $g+h\circ B$ and $h\circ B$, respectively. Although these proximity operators have no closed-form solution, it can be solved effectively from dual and primal-dual approach. It is interesting that we find three existing iterative algorithms for solving the optimization problem (\ref{three-sum}) are a special case of our proposed iterative algorithms, respectively. We prove the convergence of the proposed iterative algorithm in finite-dimensional Hilbert spaces.
 We also provide a decompose way to select iterative parameters of the Condat-Vu algorithm. The equivalence between the PDFP \cite{chenpj2016} and the PD3O \cite{yanm2016} is presented. Numerical experiments on the fused Lasso problem (\ref{fused-lasso}), the constrained total variation regularization problem (\ref{constrained-tv}) and the low-rank total variation image super-resolution problem (\ref{lrtv}) demonstrate the effectiveness of our proposed iterative algorithms. For the number of inner iterations, our numerical results confirm it could be fixed with a small number to keep the convergence of the iterative algorithms.
In the future work, we will consider to improve the proposed iterative algorithms by using line-search method without the requirement of the Lipschitz constant of $\nabla f$ and also the operator norm of $B$.

\section*{Competing interests}
The authors declare that they have no competing interests.

\section*{ACKNOWLEDGMENTS}
This work was supported by the National Natural Science Foundations
of China (11401293, 11661056, 11771198), the Natural Science
Foundations of Jiangxi Province (20151BAB211010), the China Postdoctoral Science Foundation (2015M571989)
and the Jiangxi Province Postdoctoral Science Foundation (2015KY51).

% References
%\bibliographystyle{unsrt}
%\bibliography{klreference-en,essayfirst-en} %
 %

\end{document}